\documentclass{amsart}
\usepackage{mathrsfs, amssymb}

\newcounter{counter}
\newtheorem{Lemma}[counter]{Lemma}
\newtheorem{Prop}[counter]{Proposition}
\newtheorem{Thm}[counter]{Theorem}
\newtheorem{Cor}[counter]{Corollary}

\theoremstyle{definition}
\newtheorem{Def}{Definition}
\theoremstyle{remark}
\newtheorem*{remark}{Remark}

\title[Group gradings on block-triangular matrices]{Group gradings on the Lie and Jordan algebras of block-triangular matrices}
\author{Mikhail Kochetov}
\address{Department of Mathematics and Statistics, Memorial University of Newfoundland, St. John's, NL, A1C5S7, Canada.}
\thanks{The first author was supported by Discovery Grant 2018-04883 of the Natural Sciences and Engineering Research Council (NSERC) of Canada}
\email{mikhail@mun.ca}

\author{Felipe Yukihide Yasumura}
\address{
	Department of Mathematics, Universidade Estadual de Maring\'a, Maring\'a, PR, Brazil.}
\thanks{The second author was supported by S\~ao Paulo Research Foundation (Fapesp), grant 2017/11.018-9}
\email{felipeyukihide@gmail.com}



\begin{document}

\begin{abstract}
We classify up to isomorphism all gradings by an arbitrary group $G$ on the Lie algebras of zero-trace upper block-triangular matrices over an algebraically closed field of characteristic $0$. It turns out that the support of such a grading always generates an abelian subgroup of $G$.

Assuming that $G$ is abelian, our technique also works to obtain the classification of $G$-gradings on the upper block-triangular matrices as an associative algebra, over any algebraically closed field. These gradings were originally described by A. Valenti and M. Zaicev in 2012 (assuming characteristic $0$ and $G$ finite abelian) and classified up to isomorphism by A. Borges et al. in 2018.

Finally, still assuming that $G$ is abelian, we classify $G$-gradings on the upper block-triangular matrices as a Jordan algebra, over an algebraically closed field of characteristic $0$. It turns out that, under these assumptions, the Jordan case is equivalent to the Lie case.
\end{abstract}
\maketitle

\section{Introduction}

The algebras of upper block-triangular matrices are an essential example of non-simple algebras. Moreover, viewed as Lie algebras, they are an example of the so-called parabolic subalgebras of simple Lie algebras. Group gradings on the upper triangular matrices (a Borel subalgebra) were investigated in \cite{pkfy2017}.

In this paper, we classify gradings by any abelian group $G$ on the upper block-triangular matrices, viewed as an associative, Lie or Jordan algebra, over an algebra\-ical\-ly closed field $\mathbb{F}$, which is assumed to have  characteristic $0$ in the Lie and Jordan cases. The basic idea is to show that every $G$-grading on the upper block-triangular matrices (of trace zero in the Lie case) can be extended uniquely to a grading on the full matrix algebra. However, not every $G$-grading on the full matrix algebra restricts to a grading on the upper block-triangular matrices, which leads us to consider an additional $\mathbb{Z}$-grading. In the associative case, this approach to the classification of gradings is different from the one of A. Valenti and M. Zaicev, who investigated upper triangular matrices in \cite{VaZa2007} and upper block-triangular matrices in \cite{VaZa2012} (under more restrictive assumptions than here). The Lie and Jordan cases are new. It turns out that the automorphism group of the upper block-triangular matrices, viewed as a Jordan algebra, is the same as the automorphism group of the upper block-triangular matrices of trace zero, viewed as a Lie algebra. Hence, the classifications of abelian group gradings in both cases are equivalent. The Jordan algebra of upper triangular matrices was investigated in \cite{pkfy2017Jord}.

Moreover, we prove that, in the Lie case, there is no loss of generality in assuming $G$ abelian, because the support of any group grading on the zero-trace upper block-triangular matrices generates an abelian subgroup.

The paper is structured as follows. After a brief review of terminology and relevant results on gradings in Section \ref{prelim}, we obtain a classification of gradings by abelian groups on the associative algebra of upper block-triangular matrices in Section \ref{assoc_case} (see Theorem \ref{th:main_assoc} and Corollary \ref{cor:main assoc}). In Section \ref{Lie_case}, we classify gradings on the Lie algebra of upper block-triangular matrices (Theorem \ref{th:main_Lie} and Corollary \ref{cor:main_Lie}). The center of this algebra is spanned by the identity matrix, and we actually classify gradings on the quotient modulo the center. The effect that this transition has on the classification of gradings is discussed in Section \ref{practical_iso}, the main results of which (Theorem \ref{th:main_practical} and Corollary \ref{cor:main_practical}) are quite general and may be of independent interest. Our approach to classification in Section \ref{Lie_case} follows the same lines as in the associative case. However, the Lie case is substantially more difficult, and some technical aspect is postponed until Section \ref{commut_supp}, where we also prove the commutativity of support (Theorem \ref{supp_commutativity}). Finally, the Jordan case is briefly discussed in Section \ref{Jord_case}. 

\section{Preliminaries on group gradings}\label{prelim}

Let $A$ be an arbitrary algebra over a field $\mathbb{F}$ and let $G$ be a group. We say that $A$ is \emph{$G$-graded} if $A$ is endowed with a fixed vector space decomposition,
\[
\Gamma:A=\bigoplus_{g\in G}A_g,
\]
such that $A_gA_h\subset A_{gh}$, for all $g,h\in G$. The subspace $A_g$ is called the \emph{homogeneous component of degree $g$}, and the non-zero elements $x\in A_g$ are said to be homogeneous of degree $g$. We write $\deg x=g$ for these elements. The \emph{support} of $A$ (or of $\Gamma$) is the set $\mathrm{Supp}\,A=\{g\in G\mid A_g\ne 0\}$.

A subspace $I\subset A$ is called \emph{graded} if $I=\bigoplus_{g\in G}I\cap A_g$. If $I$ is a \emph{graded ideal} {(that is, it is simultaneously an ideal and a graded subspace)}, then the quotient algebra $A/I$ inherits a natural $G$-grading. $A$ is said to be \emph{graded-simple} if $A^2\ne 0$ and $A$ does not have nonzero proper graded ideals.

If $A$ is an associative or Lie algebra, then a \emph{graded $A$-module} is an $A$-module $V$ with a fixed vector space decomposition $V=\bigoplus_{g\in G}V_g$ such that $A_g\cdot V_h\subset V_{gh}$, for all $g,h\in G$. A nonzero graded $A$-module is said to be \emph{graded-simple} if it does not have nonzero proper graded submodules. {(A \emph{graded submodule} is a submodule that is also a graded subspace.)}

Let $H$ be any group and let $\alpha:G\to H$ be a homomorphism of groups. Then $\alpha$ induces a $H$-grading, say $A=\bigoplus_{h\in H} A_h'$, on the $G$-graded algebra $A$ if we define
\[
A_h'=\bigoplus_{g\in\alpha^{-1}(h)}A_g.
\]
The $H$-grading is called the coarsening of $\Gamma$ induced by the homomorphism $\alpha$.

Let $B=\bigoplus_{g\in G}B_g$ be another $G$-graded algebra. A map $f:A\to B$ is called a \emph{homomorphism of $G$-graded algebras} if $f$ is a homomorphism of algebras and $f(A_g)\subset B_g$, for all $g\in G$. If, moreover, $f$ is an isomorphism, we call $f$ a \emph{$G$-graded isomorphism} (or an isomorphism of graded algebras), and we say that $A$ and $B$ are \emph{$G$-graded isomorphic} (or isomorphic as graded algebras). Two $G$-gradings, $\Gamma$ and $\Gamma'$, on the same algebra $A$ are \emph{isomorphic} if $(A,\Gamma)$ and $(A,\Gamma')$ are isomorphic as graded algebras.

Let $T$ be a finite abelian group and let $\sigma:T\times T\to\mathbb{F}^\times$ be a map, where $R^\times$ denotes the group of invertible elements in a ring $R$. We say that $\sigma$ is a \emph{2-cocycle} if
\[
\sigma(u,v)\sigma(uv,w)=\sigma(u,vw)\sigma(v,w),\quad\forall u,v,w\in T.
\]
The twisted group algebra $\mathbb{F}^\sigma T$ is constructed as follows: it has $\{X_t\mid t\in T\}$ as an $\mathbb{F}$-vector space basis, and multiplication is given by $X_uX_v=\sigma(u,v)X_{uv}$. It is readily seen that $\mathbb{F}^\sigma T$ is an associative algebra if and only if $\sigma$ is a 2-cocycle, which we will assume from now on. Note that $A=\mathbb{F}^\sigma T$ has a natural $T$-grading, where each homogeneous component has dimension 1, namely $A_t=\mathbb{F}X_t$, for each $t\in T$. This is an example of the so-called \emph{division grading}. A graded algebra $D$ is a \emph{graded division algebra} (or $D$ has a division grading) if every non-zero homogeneous element of $D$ is invertible.

Define $\beta:T\times T\to\mathbb{F}^\times$ by $\beta(u,v)=\sigma(u,v)\sigma(v,u)^{-1}$. Then we have 
\[
X_uX_v=\beta(u,v)X_vX_u,\quad\forall u,v\in T,
\]
and $\beta$ is an alternating bicharacter of $T$, that is, $\beta$ is multiplicative in each variable and $\beta(u,u)=1$ for all $u\in T$. If $\mathrm{char}\,\mathbb{F}$ does not divide $|T|$, then $\mathbb{F}^\sigma T$ is semisimple as an (ungraded) algebra. It follows that $\mathbb{F}^\sigma T$ is a simple algebra if and only if $\beta$ is non-degenerate. In particular, the non-degeneracy of $\beta$ implies that $|T|=\dim\mathbb{F}^\sigma T$ is a perfect square. It is known that, if $\mathbb{F}$ is algebraically closed, the isomorphism classes of matrix algebras endowed with a division grading by an abelian group $G$ are in bijection with the pairs $(T,\beta)$, where $T$ is a finite subgroup of $G$ (namely, the support of the grading) and $\beta:T\times T\to\mathbb{F}^\times$ is a non-degenerate alternating bicharacter (see e.g. \cite[Theorem 2.15]{EK2013}).

For each $n$-tuple $(g_1,\ldots,g_n)$ of elements of $G$, we can define a $G$-grading on $M_n=M_n(\mathbb{F})$ by declaring that the matrix unit $E_{ij}$ is homogeneous of degree $g_ig_j^{-1}$, for all $i$ and $j$. A grading on $M_n$ is called \emph{elementary} if it is isomorphic to one of this form. For any $g\in G$ and any permutation $\sigma\in S_n$, the $n$-tuple $(g_{\sigma(1)}g,\ldots,g_{\sigma(n)}g)$ defines an isomorphic elementary $G$-grading. Hence, an isomorphism class of elementary gradings is described by a function $\kappa:G\to\mathbb{Z}_{\ge0}$, where $g\in G$ appears exactly $\kappa(g)$ times in the $n$-tuple. Moreover, $G$ acts on these functions by translation: given $g\in G$, one defines the function $g\kappa:G\to\mathbb{Z}_{\ge0}$ by $g\kappa(x)=\kappa(g^{-1}x)$. For any $\kappa:G\to\mathbb{Z}_{\ge0}$ with finite support, we denote $|\kappa|:=\sum_{x\in G}\kappa(x)$.

If $\mathbb{F}$ is algebraically closed, then, for a fixed abelian group $G$, the isomorphism classes of $G$-gradings on $M_n$ are parametrized by the triples $(T,\beta,\kappa)$, where $T$ is a finite subgroup of $G$, $\beta:T\times T\to\mathbb{F}^\times$ is a non-degenerate alternating bicharacter, and $\kappa:G/T\to\mathbb{Z}_{\ge0}$ is a function with finite support such that $|\kappa|\sqrt{|T|}=n$. A grading in the isomorphism class corresponding to $(T,\beta,\kappa)$ can be explicitly constructed by making the following two choices:
(i) a $k$-tuple $\gamma=(g_1,\ldots,g_k)$ of elements in $G$ such that each element $x\in G/T$ occurs in the $k$-tuple $(g_1T,\ldots,g_kT)$
exactly $\kappa(x)$ times (hence $k=|\kappa|$) and (ii) a division grading on $M_\ell$ with support $T$ and bicharacter $\beta$ (hence  $|T|=\ell^2$). Since $n=k\ell$, we identify $M_n$ with $M_k\otimes M_\ell$ via Kronecker product and define a $G$-grading on $M_n$ by declaring the matrix $E_{ij}\otimes d$, with $1\le i,j\le k$, and $d$ a nonzero homogeneous element of $M_\ell$, to be of degree $g_i \deg(d) g_j^{-1}$.

Finally, two triples $(T,\beta,\kappa)$ and $(T',\beta',\kappa')$ determine the same isomorphism class if and only if $T'=T$, $\beta'=\beta$, and there exists $g\in G$ such that $\kappa'=g\kappa$ (see e.g. \cite[Theorem 2.27]{EK2013}).

\section{Associative case}\label{assoc_case}

Let $\mathbb{F}$ be {a field} and let $V$ be a finite-dimensional $\mathbb{F}$-vector space. Denote by $\mathscr{F}$ a flag of subspaces in $V$, that is
\begin{align*}
0=V_0\subsetneq V_1\subsetneq\ldots\subsetneq V_s=V.
\end{align*}
Let $n=\dim V$ and $n_i=\dim V_i/V_{i-1}$, for $i=1,2,\ldots,s$. We denote by $U(\mathscr{F})$ the set of endomorphisms of $V$ preserving the flag $\mathscr{F}$, which coincides with the upper block-triangular matrices $UT(n_1,\ldots,n_s)$ after a choice of basis of $V$ respecting the flag $\mathscr{F}$. We fix such a basis and identify $U(\mathscr{F})=UT(n_1,\ldots,n_s)\subset M_n$.

For each $m\in\mathbb{Z}$, if $|m|<s$, let $J_m\subset M_n$ denote the $m$-th block-diagonal of matrices. Formally,
\begin{align*}
J_m=\mathrm{Span}\{&E_{ij}\in M_n\mid\text{there exists $q\in\mathbb{Z}_{\ge0}$ such that}\\
&n_1+\dots+n_q<i\le n_1+\dots+n_{q+1},\mathrm{ and }\\
&n_1+\dots+n_{q+m}<j\le n_1+\dots+n_{q+m+1}\}.
\end{align*}
Setting $J_m=0$ for $|m|\ge s$, we obtain a $\mathbb{Z}$-grading $M_n=\bigoplus_{m\in\mathbb{Z}}J_m$, which is the elementary grading defined by the $n$-tuple 
\[
(\underbrace{-1,\ldots,-1}_{n_1 \text{ times}},\underbrace{-2,\ldots,-2}_{n_2\text{ times}},\ldots,\underbrace{-s,\ldots,-s}_{n_s\text{ times}}).
\] 
This grading restricts to $U(\mathscr{F})$, and we will refer to the resulting grading $U(\mathscr{F})=\bigoplus_{m\ge 0}J_m$ as the \emph{natural $\mathbb{Z}$-grading} of $U(\mathscr{F})$. The associated filtration consists of the powers of the Jacobson radical $J$ of $U(\mathscr{F})$, that is, {$\bigoplus_{i\ge m}J_i=J^m$} for all $m\ge 0$.

Let $G$ be any abelian group and denote $G^\#=\mathbb{Z}\times G$. We identify $G$ with the subset $\{0\}\times G\subset G^\#$ and $\mathbb{Z}$ with $\mathbb{Z}\times\{1_G\}\subset G^\#$. We want to find a relation between $G^\#$-gradings on $M_n$ and $G$-gradings on $U(\mathscr{F})$.

First, we note that, given any $G^\#$-grading on $M_n$, we obtain a $\mathbb{Z}$-grading on $M_n$ if we consider the coarsening induced by the projection onto the first component $G^\#\to\mathbb{Z}$.

\begin{Def}
A $G^\#$-grading on $M_n$ is said to be \emph{admissible} if $U(\mathscr{F})$ with its natural $\mathbb{Z}$-grading is a graded subalgebra of $M_n$, where $M_n$ is viewed as a $\mathbb{Z}$-graded algebra induced by the projection $G^\#\to\mathbb{Z}$.
We call an isomorphism class of $G^\#$-grading on $M_n$ \emph{admissible} if it contains an admissible grading.
\end{Def}

\begin{Lemma}\label{ind_admissible}
For any admissible $G^\#$-grading on $M_n$, {the $\mathbb{Z}$-grading induced by the projection $G^\#\to\mathbb{Z}$} has $J_m$ as its homogeneous component of degree $m$.
\end{Lemma}
\begin{proof}
From the definition of admissible grading, we know that, for any $m\ge 0$, $J_m$ is contained in the homogeneous component of degree $m$ in the induced $\mathbb{Z}$-grading on $M_n$. In particular, each $E_{ii}$ is homogeneous of degree $0$. It follows that $E_{ii}M_nE_{jj}=\mathbb{F}E_{ij}$ is a graded subspace. Hence, all $E_{ij}$ are homogeneous. Moreover, if $E_{ij}\in J_{-m}$, then $E_{ji}\in J_m$ has degree $m$, so $E_{ij}$ must have degree $-m$, since $E_{ii}=E_{ij}E_{ji}$. The result follows.
\end{proof}

Recall from Section \ref{prelim} that, {over an algebraically closed field,} any isomorphism class of $G^\#$-gradings on $M_n$ is given by a finite subgroup $T$ of $G^\#$ (hence, in fact, $T\subset G$), a non-degenerate bicharacter $\beta:T\times T\to\mathbb{F}^\times$ and a function $\kappa:G^\#/T\to\mathbb{Z}_{\ge0}$ with finite support, where $n=k\ell$, $k=|\kappa|$ and $\ell=\sqrt{|T|}$. 

\begin{Lemma}\label{lem1}
Consider a $G^\#$-grading on $M_n$ with parameters $(T,\beta,\kappa)$ and let 
\[
\gamma=\big((a_1,g_1),(a_2,g_2),\ldots,(a_k,g_k)\big)
\]
be a $k$-tuple of elements of $G^\#$ associated to $\kappa$.
Then the $\mathbb{Z}$-grading on $M_n$ induced by the projection $G^\#\to\mathbb{Z}$ is an elementary grading
defined by the $n$-tuple
\[
(\underbrace{a_1,\ldots,a_1}_{\ell\text{ times}},\underbrace{a_2,\ldots,a_2}_{\ell\text{ times}},\ldots,\underbrace{a_k,\ldots,a_k}_{\ell\text{ times}}).
\]
\end{Lemma}
\begin{proof}
We have a $G^\#$-graded isomorphism $M_n\simeq M_k\otimes M_\ell$, where $M_k$ has an elementary grading defined by $\gamma$ and $M_\ell$ has a division grading with support $T$. Since $T$ is contained in the kernel of the projection $G^\#\to\mathbb{Z}$, the factor $M_\ell$ will get the trivial induced $\mathbb{Z}$-grading. The result follows.
\end{proof}

By the previous two lemmas, the isomorphism class of $G^\#$-gradings on $M_n$ with parameters $(T,\beta,\kappa)$ is admissible if and only if $\gamma$ has the following form, up to permutation and translation by an integer:
\[
\gamma=\big((-1,g_{11}),\ldots,(-1,g_{1k_1}),(-2,g_{21}),\ldots,(-2,g_{2k_2})\ldots,(-s,g_{s1}),\ldots,(-s,g_{sk_s})\big),
\]
where $n_i=k_i\ell$ for all $i=1,2,\ldots,s$. Equivalently, this condition can be restated directly in terms of $\kappa$, regarded as a function $\mathbb{Z}\times G/T\to\mathbb{Z}_{\ge 0}$, as follows: there exist $a\in\mathbb{Z}$ and $\kappa_1,\ldots,\kappa_s:G/T\to\mathbb{Z}_{\ge0}$ with $|\kappa_i|\sqrt{|T|}=n_i$ such that 
\[
\kappa(a-i,x)=\kappa_i(x),\quad \forall i\in\{1,2,\ldots,s\},\, x\in G/T,
\]
and $\kappa(a-i,x)=0$ if $i\notin\{1,2,\ldots,s\}$.

By Lemma \ref{ind_admissible}, every admissible $G^\#$-grading 
$
M_n=\bigoplus_{(m,g)\in G^\#}A_{(m,g)}
$
restricts to a $G^\#$-grading on $U(\mathscr{F})$, hence the projection onto the second component $G^\#\to G$ induces a $G$-grading on $U(\mathscr{F})$, namely, $U(\mathscr{F})=\bigoplus_{g\in G}B_g$ where $B_g=\bigoplus_{m\ge 0}A_{(m,g)}$.

\begin{Lemma}\label{lem2}
If two admissible $G^\#$-gradings on $M_n$ are isomorphic then they induce isomorphic $G$-gradings on $U(\mathscr{F})$.
\end{Lemma}
\begin{proof}
Assume that $\psi$ is an isomorphism between two admissible $G^\#$-gradings on $M_n$. Since $\psi$ preserves degree in $G^\#$, it fixes $U(\mathscr{F})$ as a set and therefore restricts to an automorphism of $U(\mathscr{F})$. This restriction is an isomorphism between the induced $G$-gradings on $U(\mathscr{F})$.
\end{proof}

Now we want to go back from $G$-gradings on $U(\mathscr{F})$ to $G^\#$-gradings on $M_n$. First note that the $G$-gradings on $U(\mathscr{F})$ obtained as above are not arbitrary, but satisfy the following:

\begin{Def}
We say that a $G$-grading on $U(\mathscr{F})$ is \emph{in canonical form} if, for each $m\in\{0,1,\ldots,s-1\}$, the subspace $J_m$ is $G$-graded.
\end{Def}

In other words, a $G$-grading $\Gamma:U(\mathscr{F})=\bigoplus_{g\in G}B_g$ is in canonical form if and only if it is compatible with the natural $\mathbb{Z}$-grading on $U(\mathscr{F})$. If this is the case, we obtain a $G^\#$-grading on $U(\mathscr{F})$ by taking $J_m\cap B_g$ as the homogeneous component of degree $(m,g)$. We want to show that this $G^\#$-grading uniquely extends to $M_n$. 

To this end, let us look more closely at the automorphism group of $U(\mathscr{F})$. 
We denote by $\mathrm{Int}(x)$ the inner automorphism $y\mapsto xyx^{-1}$ determined by an invertible element $x$.

\begin{Lemma}\label{aut}
	$\mathrm{Aut}(U(\mathscr{F}))\simeq\left\{\psi\in\mathrm{Aut}(M_n)\mid\psi(U(\mathscr{F}))=U(\mathscr{F})\right\}$.
\end{Lemma}

\begin{proof}
	It is proved in \cite[Corollary 5.4.10]{Cheung} that
	\begin{align*}
	\mathrm{Aut}(U(\mathscr{F}))=\{\mathrm{Int}(x)\mid x\in U(\mathscr{F})^\times\}.
	\end{align*}
	On the other hand, every automorphism of the matrix algebra is inner, so let $y\in M_n^\times$ and assume $yU(\mathscr{F})y^{-1}=U(\mathscr{F})$. Then, by the description of $\mathrm{Aut}(U(\mathscr{F}))$ above, we can find $x\in U(\mathscr{F})^\times$ such that
	\begin{align*}
	\mathrm{Int}(x)\mid_{U(\mathscr{F})}=\mathrm{Int}(y)\mid_{U(\mathscr{F})}.
	\end{align*}
	It follows that $xy^{-1}$ commutes with all elements of $U(\mathscr{F})$. Hence $yx^{-1}=\lambda\,{1}$, for some $\lambda\in\mathbb{F}^\times$, and $y=\lambda x\in U(\mathscr{F})^\times$.
\end{proof}

Assume for a moment that {$\mathbb{F}$ is algebraically closed and $\mathrm{char}\,\mathbb{F}=0$. Since $G$ is abelian,} it is well known that $G$-gradings on a finite-dimensional algebra $A$ are equivalent to actions of the algebraic group $\widehat{G}:=\mathrm{Hom}_\mathbb{Z}(G,\mathbb{F}^\times)$ by automorphisms of $A$, that is, homomorphisms of algebraic groups $\widehat{G}\to\mathrm{Aut}(A)$ (see, for example, \cite[\S 1.4]{EK2013}). The homomorphism $\eta_\Gamma:\widehat{G}\to\mathrm{Aut}(A)$ corresponding to a grading $\Gamma:A=\bigoplus_{g\in G}A_g$ is defined by $\eta_\Gamma(\chi)(x)=\chi(g)x$ for all $\chi\in\widehat{G}$, $g\in G$ and $x\in A_g$.

By Lemma \ref{aut}, we have 
\begin{align*}
\mathrm{Aut}\left(U(\mathscr{F})\right)\simeq\mathrm{Stab}_{\mathrm{Aut}(M_n)}(U(\mathscr{F}))\subset\mathrm{Aut}(M_n),
\end{align*}
hence, if {$\mathbb{F}$ is algebraically closed and} $\mathrm{char}\,\mathbb{F}=0$, we obtain the desired unique extension of gradings from $U(\mathscr{F})$ to $M_n$. To generalize this result to {arbitrary $\mathbb{F}$}, we can use group schemes instead of groups. Recall that an \emph{affine group scheme} over a field $\mathbb{F}$ is a representable functor from the category $\mathrm{Alg}_\mathbb{F}$ of unital commutative associative $\mathbb{F}$-algebras to the category of groups (see e.g. \cite{Waterhouse} or \cite[Appendix A]{EK2013}). For example, the \emph{automorphism group scheme} of a finite-dimensional algebra $A$ is defined by 
\[
\mathbf{Aut}(A)(R):=\mathrm{Aut}_R(A\otimes R),\quad\forall R\in\mathrm{Alg}_\mathbb{F}.
\]
Another example of relevance to us is $\mathbf{GL}_1(A)$, for a finite-dimensional associative algebra $A$, defined by $\mathbf{GL}_1(A)(R):=(A\otimes R)^\times$. (In particular, $\mathbf{GL}_1(M_n)=\mathbf{GL}_n$.) Note that we have a homomorphism $\mathrm{Int}:\mathbf{GL}_1(A)\to\mathbf{Aut}(A)$.
 
If $G$ is an abelian group, then the group algebra $\mathbb{F}G$ is a commutative Hopf algebra, so it represents an affine group scheme, which is the scheme version of the character group $\widehat{G}$. It is denoted by $G^D$ and given by $G^D(R)=\mathrm{Hom}_{\mathbb{Z}}(G,R^\times)$. In particular, $G^D(\mathbb{F})=\widehat{G}$. If we have a $G$-grading $\Gamma$ on $A$, then we can define a homomorphism of group schemes $\eta_\Gamma:G^D\to\mathbf{Aut}(A)$ by generalizing the formula in the case of $\widehat{G}$: $(\eta_\Gamma)_R(\chi)(x\otimes r)=x\otimes \chi(g)r$ for all $R\in\mathrm{Alg}_\mathbb{F}$, $\chi\in G^D(R)$, $r\in R$, $g\in G$ and $x\in A_g$. In this way, over an arbitrary field, $G$-gradings on $A$ are equivalent to homomorphisms of group schemes $G^D\to\mathbf{Aut}(A)$. 

\begin{Lemma}\label{aut_scheme}
Over an arbitrary field, $\mathbf{Aut}(U(\mathscr{F}))$ is a quotient of $\mathbf{GL}_1(U(\mathscr{F}))$, and $\mathbf{Aut}\left(U(\mathscr{F})\right)\simeq\mathbf{Stab}_{\mathbf{Aut}(M_n)}(U(\mathscr{F}))$ via the restriction map.
\end{Lemma}

\begin{proof}
We claim that the homomorphism $\mathrm{Int}:\mathbf{GL}_1(U(\mathscr{F}))\to\mathbf{Aut}(U(\mathscr{F}))$ is a quotient map {(in the sense of affine group schemes, see e.g. \cite[Chapter 15]{Waterhouse} or \cite[\S A.2]{EK2013})}. Since $\mathbf{GL}_1(U(\mathscr{F}))$ is smooth, it is sufficient to verify that (i) the group homomorphism $\mathrm{Int}:(U(\mathscr{F})\otimes\overline{\mathbb{F}})^\times\to\mathrm{Aut}_{\overline{\mathbb{F}}}(U(\mathscr{F})\otimes\overline{\mathbb{F}})$ is surjective, where $\overline{\mathbb{F}}$ is the algebraic closure of $\mathbb{F}$, and (ii) the Lie homomorphism $\mathrm{ad}:U(\mathscr{F})\to\mathrm{Der}(U(\mathscr{F}))$ is surjective (see e.g. \cite[Corollary A.49]{EK2013}). But (i) is satisfied by Corollary 5.4.10 in \cite{Cheung}, mentioned above, and (ii) is satisfied by Theorem 2.4.2 in the same work.

Since the homomorphism $\mathrm{Int}:\mathbf{GL}_1(U(\mathscr{F}))\to\mathbf{Aut}(U(\mathscr{F}))$ factors through the restriction map $\mathbf{Stab}_{\mathbf{Aut}(M_n)}(U(\mathscr{F}))\to\mathbf{Aut}\left(U(\mathscr{F})\right)$, it follows that this latter is also a quotient map. But its kernel is trivial, because the corresponding restriction maps for the group $\mathrm{Stab}_{\mathrm{Aut}_{\overline{\mathbb{F}}}(M_n(\overline{\mathbb{F}}))}(U(\mathscr{F})\otimes\overline{\mathbb{F}})$ and Lie algebra $\mathrm{Stab}_{\mathrm{Der}(M_n)}(U(\mathscr{F}))$ are injective (see e.g. \cite[Theorem A.46]{EK2013}).
\end{proof}

Coming back to a $G$-grading $\Gamma$ on $U(\mathscr{F})$ in canonical form, we conclude by Lemma \ref{aut_scheme} that the corresponding $G^\#$-grading on $U(\mathscr{F})$ extends to a unique $G^\#$-grading $\Gamma^\#$ on $M_n$. By construction,  $\Gamma^\#$ is admissible and induces the original grading $\Gamma$ on $U(\mathscr{F})$. It is also clear that $\Gamma^\#$ is uniquely determined by these properties. Thus, we have a bijection between admissible $G^\#$-gradings on $M_n$ and $G$-gradings on $U(\mathscr{F})$ in canonical form.

\begin{Lemma}\label{can_assoc}
For any $G$-grading on $U(\mathscr{F})$, there exists an isomorphic $G$-grading in canonical form.
\end{Lemma}

\begin{proof}
It follows from Lemma \ref{aut_scheme} that the Jacobson radical $J=\bigoplus_{m>0}J_m$ of $U(\mathscr{F})$ is stabilized by $\mathbf{Aut}(U(\mathscr{F}))$. Hence, $J$ is a $G$-graded ideal, so the proof of \cite[Lemma 1]{y2018} shows that, in fact, there exists an isomorphic grading such that each block is a graded subspace.
\end{proof}

\begin{Lemma}\label{iso_grad1}
If two $G$-gradings, $\Gamma_1$ and $\Gamma_2$, on $U(\mathscr{F})$ are in canonical form and isomorphic to one another, then there exists a block-diagonal matrix $x\in U(\mathscr{F})^\times$ such that $\psi_0=\mathrm{Int}(x)$ is an isomorphism between $\Gamma_1$ and $\Gamma_2$.
\end{Lemma}

\begin{proof}
Let $\psi=\mathrm{Int}(y)$ be an isomorphism between $\Gamma_1$ and $\Gamma_2$. Write $y=(y_{ij})_{1\le i\le j\le s}$ in blocks and let $x=\mathrm{diag}(y_{11},\ldots,y_{ss})$. Then $x$ is invertible, so let $\psi_0=\mathrm{Int}(x)$.

Fix $m\in\{0,1,\ldots,s-1\}$ and let $a\in J_m$ be $G$-homogeneous with respect to $\Gamma_1$. Since $J^m=J_m\oplus J^{m+1}$, we can uniquely write $\psi(a)=b+c$, where $b\in J_m$ and $c\in J^{m+1}$. Since $\Gamma_2$ is in canonical form, $J_m$ and $J^{m+1}$ are $G$-graded subspaces with respect to $\Gamma_2$. Since $\psi$ preserves $G$-degree, it follows that $b$ and $c$ are $G$-homogeneous elements with respect to $\Gamma_2$ of the same $G$-degree as $a$ with respect to $\Gamma_1$. Finally, note that $\psi_0(a)=b$. Since $m$ and $a$ were arbitrary, we have shown that $\psi_0$ is an isomorphism between $\Gamma_1$ and $\Gamma_2$. 
\end{proof}

Now we can prove the converse of Lemma \ref{lem2}.

\begin{Lemma}\label{iso_grad2}
If two admissible $G^\#$-gradings on $M_n$ induce isomorphic $G$-gradings on $U(\mathscr{F})$, then they are isomorphic.
\end{Lemma}

\begin{proof}
Let $\Gamma_1$ and $\Gamma_2$ be two isomorphic $G$-gradings on $U(\mathscr{F})$ obtained from two $G^\#$-gradings on $M_n$, $\Gamma_1^\#$ and $\Gamma_2^\#$, respectively. For $i=1,2$, let $\eta_i:(G^\#)^D\to\mathbf{Aut}(M_n)$ be the action corresponding to $\Gamma_i^\#$. Consider also the restriction $\Gamma_i'$ of $\Gamma_i^\#$ to $U(\mathscr{F})$ and the corresponding action  $\eta_i':(G^\#)^D\to\mathbf{Aut}(U(\mathscr{F}))$. By Lemma \ref{iso_grad1}, we can find an isomorphism $\psi_0=\mathrm{Int}(x)$ between $\Gamma_1$ and $\Gamma_2$, where $x$ is block-diagonal. Such $\psi_0$ preserves the natural $\mathbb{Z}$-grading, so it is actually an isomorphism between the $G^\#$-gradings $\Gamma_1'$ and $\Gamma_2'$. Hence, $\psi_0\eta_1'(\chi)=\eta_2'(\chi)\psi_0$ for all $\chi\in(G^\#)^D(R)$ and all $R\in\mathrm{Alg}_\mathbb{F}$. By Lemma \ref{aut_scheme}, this implies $\psi_0\eta_1(\chi)=\eta_2(\chi)\psi_0$ for all $\chi\in(G^\#)^D(R)$, which means $\psi_0$ is an isomorphism between $\Gamma_1^\#$ and $\Gamma_2^\#$.
\end{proof}

We summarize the results of this section:

\begin{Thm}\label{th:main_assoc}
{Over an arbitrary field,} the mapping of an admissible $G^\#$-grading on $M_n$ to a $G$-grading on $U(\mathscr{F})$, given by restriction and coarsening, yields a bijection between the admissible isomorphism classes of $G^\#$-gradings on $M_n$ and the isomorphism classes of $G$-gradings on $U(\mathscr{F})$.\qed
\end{Thm}

{If $\mathbb{F}$ is algebraically closed, then the} admissible isomorphism classes of $G^\#$-gradings on $M_n$ can be parametrized by the triples $(T,\beta,(\kappa_1,\ldots,\kappa_s))$, where $T\subset G$ is a finite subgroup, $\beta:T\times T\to\mathbb{F}^\times$ is a non-degenerate alternating bicharacter and $\kappa_i:G/T\to\mathbb{Z}_{\ge0}$ are functions with finite support such that $|\kappa_i|\sqrt{|T|}=n_i$, for each $i=1,2,\ldots,s$. Hence, isomorphism classes of $G$-gradings on $U(\mathscr{F})$ are parametrized by the same triples.

Choosing, for each $\kappa_i$, a $k_i$-tuple $\gamma_i$ of elements of $G$, where $k_i=|\kappa_i|$, we reproduce the description of $G$-gradings on $U(\mathscr{F})$ originally obtained in \cite{VaZa2012}. Note, however, that we do not need to assume that $G$ is finite, nor $\mathrm{char}\,\mathbb{F}=0$. Also note that we have a description not only of $G$-gradings but of their isomorphism classes, which gives an alternative proof of the following result first established in \cite[Corollary 4]{BFD2018}:

\begin{Cor}\label{cor:main assoc}
Two $G$-gradings on $U(\mathscr{F})$, determined by $(T,\beta,(\kappa_1,\ldots,\kappa_s))$ and $(T',\beta',(\kappa_1',\ldots,\kappa_s'))$, are isomorphic if and only if $T'=T$, $\beta'=\beta$ and there exists $g\in G$ such that $\kappa_i'=g\kappa_i$, for all $i=1,2,\ldots,s$.\qed
\end{Cor}

\section{Lie case}\label{Lie_case}

Now we turn our attention to $U(\mathscr{F})^{(-)}$, that is, $U(\mathscr{F})$ viewed as a Lie algebra with respect to the commutator $[x,y]=xy-yx$. {Since we will be working with Lie and associative products at the same time, we will always indicate the former by brackets and keep using juxtaposition for the latter.} We assume that the grading group $G$ is abelian and the ground field $\mathbb{F}$ is algebraically closed of characteristic $0$, and follow the same approach as in the associative case. 

Denote by $\tau$ the flip along the secondary diagonal on $M_n${, that is, $\tau(E_{ij})=E_{n-j+1,n-i+1}$, for all matrix units $E_{ij}\in M_n$}. Note that $U(\mathscr{F})^\tau=U(\mathscr{F})$ if and only if $n_i=n_{s-i+1}$ for all $i=1,2,\ldots,\lfloor\frac{s}2\rfloor$. Let 
\[
U(\mathscr{F})_0=\{x\in U(\mathscr{F})\mid\mathrm{tr}(x)=0\},
\] 
which is a Lie subalgebra of $U(\mathscr{F})^{(-)}$. Moreover, $U(\mathscr{F})^{(-)}=U(\mathscr{F})_0\oplus\mathbb{F}{1}$, where ${1}\in U(\mathscr{F})$ is the identity matrix. 
The center $\mathfrak{z}(U(\mathscr{F})^{(-)})=\mathbb{F}{1}$ is always graded, so ${1}$ is a homogeneous element. If we change its degree arbitrarily, we obtain a new well-defined grading, which is not isomorphic to the original one, but will induce the same grading on $U(\mathscr{F})^{(-)}/\mathbb{F}{1}\simeq U(\mathscr{F})_0$ (compare with  \cite[Definition 6]{pkfy2017}). It turns out that, up to isomorphism, a $G$-grading on $U(\mathscr{F})^{(-)}$ is determined by the induced $G$-grading on $U(\mathscr{F})_0$ and the degree it assigns to the identity matrix (see Corollary \ref{reduction_to_trace_0} in Section \ref{practical_iso}). Conversely, any $G$-grading on $U(\mathscr{F})_0$ extends to $U(\mathscr{F})^{(-)}=U(\mathscr{F})_0\oplus\mathbb{F}{1}$ by defining the degree of ${1}$ arbitrarily. Thus, we have a bijection between the isomorphism classes of $G$-gradings on $U(\mathscr{F})^{(-)}$ and the pairs consisting of an isomorphism class of $G$-gradings on $U(\mathscr{F})_0$ and an element of $G$.

We start by computing the automorphism group of $U(\mathscr{F})_0$. To this end, we will use the following description of the automorphisms of $\mathrm{Aut}(U(\mathscr{F})^{(-)})$, which was proved in \cite{MaSo1999} for the field of complex numbers. 

\begin{Thm}[{\cite[Theorem 4.1.1]{Cecil}}]\label{aut_cecil}
Let $\phi$ be an automorphism of $U(\mathscr{F})^{(-)}$, and assume $\mathrm{char}\,\mathbb{F}=0$ or $\mathrm{char}\,\mathbb{F}>3$. Then there exist $p,d\in U(\mathscr{F})$, with $p$ invertible and $d$ block-diagonal, such that one of the following holds:
\begin{enumerate}
\item $\phi(x)=pxp^{-1}+\mathrm{tr}(xd){1}$, for all $x\in U(\mathscr{F})$, or
\item $\phi(x)=-px^\tau p^{-1}+\mathrm{tr}(xd){1}$, for all $x\in U(\mathscr{F})$.\qed
\end{enumerate}
\end{Thm}

\begin{remark}\label{antiaut}
Case (2) in the previous theorem occurs if and only if $U(\mathscr{F})$ is invariant under $\tau$, that is, $n_i=n_{s-i+1}$ for all $i$. It follows that $U(\mathscr{F})$ admits an anti-automorphism only under this condition. Indeed, if $\psi$ is an anti-automorphism of $U(\mathscr{F})$, then $-\psi$ is a Lie automorphism of $U(\mathscr{F})$. Hence, by Theorem \ref{aut_cecil}, we have $-\psi(x)=pxp^{-1}+\mathrm{tr}(xd){1}$ for all $x\in U(\mathscr{F})$ or $n_i=n_{s-i+1}$ for all $i$. However, the first possibility cannot occur if $n>2$, since it would imply that the composition $\psi\,\mathrm{Int}(p^{-1})$, which maps $x\mapsto-x+\mathrm{tr}(xd'){1}$ where $d'$ is the block-diagonal part of $-pdp^{-1}$, is an anti-automorphism of $U(\mathscr{F})$, but this is easily seen not to be the case. (Of course, if $n=2$ then we have  $n_i=n_{s-i+1}$ for all $i$.) 
\end{remark}

As a consequence, we obtain the following analog of Lemma \ref{aut}. {(As usual, the symbol $\rtimes$ denotes a semidirect product in which the second factor acts on the first.)} 

\begin{Lemma}\label{aut_Lie}
If $n>2$ and $n_i=n_{s-i+1}$ for all $i$, then
\[
\mathrm{Aut}(U(\mathscr{F})_0)\simeq\{\mathrm{Int}(x)\mid x\in U(\mathscr{F})^\times\}\rtimes\langle-\tau\rangle;
\]
otherwise, $\mathrm{Aut}(U(\mathscr{F})_0)\simeq\{\mathrm{Int}(x)\mid x\in U(\mathscr{F})^\times\}$. In both cases,
\[
\mathrm{Aut}(U(\mathscr{F})_0)\simeq\mathrm{Stab}_{\mathrm{Aut}(\mathfrak{sl}_n)}(U(\mathscr{F})_0).
\]
\end{Lemma}

\begin{proof}
Let $\psi\in\mathrm{Aut}(U(\mathscr{F})_0)$. We extend $\psi$ to an automorphism $\phi$ of $U(\mathscr{F})^{(-)}$ by setting 
$\phi({1})={1}$. By the previous result, $\phi$ must have one of two possible forms. Assume it is the first one: 
\[
\phi(x)=pxp^{-1}+\mathrm{tr}(xd){1},\quad\forall x\in U(\mathscr{F}).
\]
But as $U(\mathscr{F})_0$ is an invariant subspace for $\phi$, we see that, for all $x\in U(\mathscr{F})_0$,
\[
0=\mathrm{tr}(\phi(x))=\mathrm{tr}(pxp^{-1}+\mathrm{tr}(xd){1})=n\,\mathrm{tr}(xd).
\]
Therefore, $\mathrm{tr}(xd)=0$ and hence $\psi(x)=\phi(x)=pxp^{-1}$, for all $x\in U(\mathscr{F})_0$, so $\psi=\mathrm{Int}(p)$. The same argument applies if $\phi$ has the second form. 
Note that, for $n=2$, the second form reduces to the first on $UT(1,1)_0$, since $-\tau$ coincides with $\mathrm{Int}(p)$ on $\mathfrak{sl}_2$, where $p=\mathrm{diag}(1,-1)$.
On the other hand, for $n>2$, the two forms do not overlap, since the action of $-\tau$ differs already on the set of zero-trace diagonal matrices from the action of any inner automorphism.
We conclude the proof in the same way as for Lemma \ref{aut}.
\end{proof}

Let $G$ be an abelian group and define $G^\#=\mathbb{Z}\times G$. 
Similarly to the associative case, we want to relate $G$-gradings on $U(\mathscr{F})_0$ and $G^\#$-gradings on $\mathfrak{sl}_n$, since for the latter a classification of group gradings is known \cite{BK2010} (see also \cite[Chapter 3]{EK2013}). 

Recall that $J_m$ stands for the $m$-th block-diagonal of matrices. We consider again the \emph{natural $\mathbb{Z}$-grading} on $U(\mathscr{F})_0$: its homogeneous component of degree $m\in\mathbb{Z}$ is $J_m\cap U(\mathscr{F})_0$ if $0\le m<s$ and $0$ otherwise. We say that a $G$-grading on $U(\mathscr{F})_0$ is in \emph{canonical form} if, for each $m\in\{0,\ldots,s-1\}$, the subspace $J_m\cap U(\mathscr{F})_0$ is $G$-graded. A $G^\#$-grading on $\mathfrak{sl}_n$ is said to be \emph{admissible} if the coarsening induced by the projection $G^\#\to\mathbb{Z}$ has $U(\mathscr{F})_0$, with its natural $\mathbb{Z}$-grading, as a graded subalgebra. An isomorphism class of $G^\#$-grading on $\mathfrak{sl}_n$ is called \emph{admissible} if it contains an admissible grading.

Since any $\mathbb{Z}$-grading on $\mathfrak{sl}_n$ is the restriction of a unique $\mathbb{Z}$-grading on the associative algebra $M_n$, Lemma \ref{ind_admissible} still holds if we replace $M_n$ by $\mathfrak{sl}_n$. Therefore, every admissible $G^\#$-grading on $\mathfrak{sl}_n$ restricts to $U(\mathscr{F})_0$ and, by means of the projection $G^\#\to G$, yields a $G$-grading on $U(\mathscr{F})_0$, which is clearly in canonical form. 
Conversely, thanks to Lemma \ref{aut_Lie}, if a $G$-grading on $U(\mathscr{F})_0$ is in canonical form then it comes from a unique admissible $G^\#$-grading on $\mathfrak{sl}_n$ in this way. Therefore, similarly to the associative case, we obtain a bijection between admissible $G^\#$-grading on $\mathfrak{sl}_n$ and $G$-gradings on $U(\mathscr{F})_0$ in canonical form. 

The following result is technical and will be proved in Section \ref{commut_supp}:

\begin{Lemma}\label{can_Lie}
For any $G$-grading on $U(\mathscr{F})_0$, there exists an isomorphic $G$-grading in canonical form.
\end{Lemma}

Clearly, as in Lemma \ref{lem2}, if two admissible $G^\#$-gradings on $\mathfrak{sl}_n$ are isomorphic then they induce isomorphic $G$-gradings on $U(\mathscr{F})_0$. The converse is established by the same argument as Lemma \ref{iso_grad2}, using the following analog of Lemma \ref{iso_grad1}:

\begin{Lemma}
If two $G$-gradings, $\Gamma_1$ and $\Gamma_2$, on $U(\mathscr{F})_0$ are in canonical form and isomorphic to one another, then there exists an isomorphism $\psi_0$ between $\Gamma_1$ and $\Gamma_2$ of the form $\psi_0=\mathrm{Int}(x)$ or $\psi_0=-\mathrm{Int}(x)\tau$ where the matrix $x\in U(\mathscr{F})^\times$ is block-diagonal.
\end{Lemma}

\begin{proof}
Let $\psi$ be an isomorphism between $\Gamma_1$ and $\Gamma_2$. If $\psi=\mathrm{Int}(y)$ then we are in the situation of the proof of Lemma \ref{iso_grad1}. If $\psi=-\mathrm{Int}(y)\tau$ then the same proof still works because all subspaces $J_m$ are invariant under $\tau$.
\end{proof}

In summary:

\begin{Thm}\label{th:main_Lie}
The mapping of an admissible $G^\#$-grading on $\mathfrak{sl}_n$ to a $G$-grading on $U(\mathscr{F})_0$, given by restriction and coarsening, yields a bijection between the admissible isomorphism classes of $G^\#$-gradings on $\mathfrak{sl}_n$ and the isomorphism classes of $G$-gradings on $U(\mathscr{F})_0$.\qed
\end{Thm}

There are two families of gradings on $\mathfrak{sl}_n$, $n>2$, namely, Type I and Type II. (Only Type I exists for $n=2$.) Their isomorphism classes are stated in Theorem 3.53 of \cite{EK2013}, but we will use Theorem 45 of \cite{BKE2018}, which is equivalent but uses more convenient parameters. 

By definition, a $G^\#$-grading of Type I is a restriction of a $G^\#$-grading on the associative algebra $M_n$, so it is parametrized by $(T,\beta,\kappa)$, where, as in Section \ref{assoc_case}, $T\subset G$ is a finite group, $\beta:T\times T\to\mathbb{F}^\times$ is a non-degenerate alternating bicharacter and $\kappa:\mathbb{Z}\times G/T\to\mathbb{Z}_{\ge0}$ is a function with finite support satisfying $|\kappa|\sqrt{|T|}=n$. 

For a Type II grading, there is a unique element $f\in G^\#$ of order $2$ (hence, in fact, $f\in G$), called the \emph{distinguished element}, such that the coarsening induced by the natural homomorphism $G^\#\to G^\#/\langle f\rangle$ is a Type I grading. The parametrization of Type II gradings depends on the choice of character $\chi$ of $G^\#$ satisfying $\chi(f)=-1$. So, we fix $\chi\in\widehat{G}$ with $\chi(f)=-1$ and extend it trivially to the factor $\mathbb{Z}$. Then, the parameters of a Type II grading are a finite subgroup $T\subset G^\#$ (hence $T\subset G$) containing $f$, an alternating bicharacter $\beta:T\times T\to\mathbb{F}^\times$ with radical $\langle f\rangle$ (so, $\beta$ determines the distinguished element $f$), an element $g^\#_0\in G^\#$, and a function $\kappa:\mathbb{Z}\times G/T\to\mathbb{Z}_{\ge0}$ with finite support satisfying $|\kappa|\displaystyle\sqrt{|T|/2}=n$. These parameters are required to satisfy some additional conditions, as follows. 

To begin with, for a Type II grading, $T$ must be $2$-elementary. Its Type I coarsening is a grading by $G^\#/\langle f\rangle\simeq\mathbb{Z}\times\overline{G}$ with parameters $(\overline{T},\bar{\beta},\kappa)$, where $\overline{T}:=T/\langle f\rangle$ is a subgroup of $\overline{G}:=G/\langle f\rangle$, $\bar{\beta}:\overline{T}\times\overline{T}\to\mathbb{F}^\times$ is the non-degenerate bicharacter induced by $\beta$, and $\kappa$ is now regarded as a function on $\mathbb{Z}\times\overline{G}/\overline{T}\simeq\mathbb{Z}\times G/T$. 

Since $T$ is $2$-elementary, $\beta$ can only take values $\pm 1$ and $\ell:=\sqrt{|T|/2}$ is a power of $2$. If one uses Kronecker products of Pauli matrices (of order $2$) to construct a division grading on $M_\ell$ with support $\overline{T}$ and bicharacter $\bar{\beta}$, then the transposition will preserve degree and thus become an involution on the resulting graded division algebra $D$. The choice of such an involution is arbitrary, and it will be convenient for our purposes to use $\tau$, which also preserves degree. Since all homogeneous components of $D$ are $1$-dimensional, we have 
\[
(X_{\bar{t}})^\tau=\bar{\eta}(\bar{t})X_{\bar{t}},\quad\forall\bar{t}\in\overline{T},\,X_{\bar{t}}\in D_{\bar{t}},
\]
where $\bar{\eta}:\overline{T}\to\{\pm 1\}$ satisfies $\bar{\eta}(\bar{u}\bar{v})=\bar{\beta}(\bar{u},\bar{v})\bar{\eta}(\bar{u})\bar{\eta}(\bar{v})$ for all $\bar{u},\bar{v}\in\overline{T}$. If we regard $\bar{\eta}$ and $\bar{\beta}$ as maps of vector spaces over the field of two elements, this equation means that $\bar{\eta}$ is a quadratic form with polarization $\bar{\beta}$. {Define a quadratic form $\eta:T\to\{\pm 1\}$ with polarization $\beta$ by $\eta(t)=\chi(t)\bar{\eta}(\bar{t})$, where $\bar{t}$ denotes the image of $t\in T$ in the quotient group $\overline{T}$.} 

Recall that a concrete $G^\#/\langle f\rangle$-grading with parameters $(\overline{T},\bar{\beta},\kappa)$ is constructed by selecting a $k$-tuple  of elements of $G^\#/\langle f\rangle$, as directed by $\kappa$, to get an elementary grading on $M_k$, where $k=|\kappa|$, and identifying $M_n\simeq M_k\otimes D$ via Kronecker product. The remaining parameter $g^\#_0$ can then be used, together with the chosen involution $\tau$ on $D$, to define an anti-automorphism $\varphi$ on $M_n$ by the formula
\[
\varphi(X)=\Phi^{-1}X^\tau\Phi,\quad\forall X\in M_n,
\] 
where the matrix $\Phi\in M_k\otimes D\simeq M_k(D)$ is constructed in such a way that $\varphi^2$ acts on $M_n$ in exactly the same way as $\chi^2$, which acts on $M_n$ because it can be regarded as a character on $G^\#/\langle f\rangle$ (since $\chi^2(f)=1$) and $M_n$ is a $G^\#/\langle f\rangle$-graded algebra. As a result, we can split each homogeneous component of the $G^\#/\langle f\rangle$-grading on $M_n$ into (at most $2$) eigenspaces of $\varphi$ so that the action of $\chi$ on the resulting $G^\#$-graded algebra $M_n^{(-)}$ coincides with the automorphism $-\varphi$. Finally, the restriction of this $G^\#$-grading to $\mathfrak{sl}_n$ is a $G^\#$-grading of Type II with parameters $(T,\beta,g^\#_0,\kappa)$.

In order to construct $\Phi$, two conditions must be met: 
\begin{enumerate}
\item[(i)] $\kappa$ is \emph{$g^\#_0$-balanced} in the sense that $\kappa(x)=\kappa((g_0^{\#})^{-1}x^{-1})$ for all $x\in\mathbb{Z}\times G/T$ (where the inverse in $\mathbb{Z}$ is understood with respect to addition);
\item[(ii)] $\kappa(g^\#T)$ is even whenever $g_0^\#(g^\#)^2\in T$ and $\eta(g^\#_0(g^\#)^2)=-1$ for some $g^\#\in G^\#$. 
\end{enumerate}
Such a matrix $\Phi\in M_k(D)$ is given explicitly by Equations (3.29) and (3.30) in \cite{EK2013}, but in relation to the usual transposition. Since we are using $\tau$, the order of the $k$ rows has to be reversed and the entries in $D$ chosen in accordance with the above quadratic form $\bar{\eta}$ rather than the quadratic form in \cite{EK2013}. It will also be convenient in our situation to order the $k$-tuple associated to $\kappa$ in a different way, as will be described below. 

We are only interested in admissible isomorphism classes of $G^\#$-gradings on $\mathfrak{sl}_n$. If $n=2$, the isomorphism condition for (Type I) gradings is the same as in the associative case: all translations of $\kappa$ determine isomorphic gradings. If $n>2$, however, one isomorphism class of Type I gradings on $\mathfrak{sl}_n$ can consist of one or two isomorphism classes of gradings on $M_n$, because $(T,\beta,\kappa)$ and $(T,\beta^{-1},\bar{\kappa})$  determine isomorphic gradings on $\mathfrak{sl}_n$, where the function $\bar{\kappa}:\mathbb{Z}\times G/T\to\mathbb{Z}_{\ge0}$ is defined by $\bar{\kappa}(i,x):=\kappa(-i,x^{-1})$. Hence, the isomorphism class of $G^\#$-gradings of Type I with parameters $(T,\beta,\kappa)$ is admissible if and only if at least one of the functions $\kappa$ and $\bar{\kappa}$ has the form described after Lemma \ref{lem1}. Assuming it is $\kappa$, there must exist $a\in\mathbb{Z}$ and functions $\kappa_1,\ldots,\kappa_s:G/T\to\mathbb{Z}_{\ge0}$ with $|\kappa_i|\sqrt{|T|}=n_i$, such that 
\begin{equation}\label{ref_kappa_seq}
\kappa(a-i,x)=\kappa_i(x),\quad\forall i\in\{1,2,\ldots,s\},\,x\in G/T,
\end{equation}
and $\kappa(a-i,x)=0$ if $i\not\in\{1,2,\ldots,s\}$. Then $\bar{\kappa}$ can be expressed in the same form, but with the function $\bar{\kappa}_i(x):=\kappa_i(x^{-1})$ playing the role of $\kappa_{s-i+1}$ for each $i$. Thus, the isomorphism classes of $G$-gradings of Type I on $U(\mathscr{F})_0$ are parametrized by $(T,\beta,(\kappa_1,\ldots,\kappa_s))$, and, if $n_i=n_{s-i+1}$ for all $i$, then $(T,\beta,(\kappa_1,\ldots,\kappa_s))$ and $(T,\beta^{-1},(\bar{\kappa}_s,\ldots,\bar{\kappa}_1))$ determine isomorphic $G$-gradings on $U(\mathscr{F})_0$. 

Now consider the isomorphism class of Type II gradings on $\mathfrak{sl}_n$ ($n>2$) with parameters $(T,\beta,g^\#_0,\kappa)$. Admissibility is a condition on the $\mathbb{Z}$-grading induced by the projection $G^\#\to\mathbb{Z}$, which factors through the natural homomorphism $G^\#\to G^\#/\langle f\rangle$. So, for this isomorphism class to be admissible, it is necessary and sufficient for $\kappa$ to have the form given by Equation \eqref{ref_kappa_seq}, but with $|\kappa_i|\sqrt{|T|/2}=n_i$.

\begin{Lemma}\label{balance_kappa_i}
If $g^\#_0=(a_0,g_0)$ and $\kappa$ is given by Equation \eqref{ref_kappa_seq}, then $\kappa$ is $g^\#_0$-balanced if and only if $a_0=s+1-2a$ and $\kappa_i(x)=\kappa_{s-i+1}(g_0^{-1}x^{-1})$ for all $x\in G/T$ and all $i$.
\end{Lemma}

\begin{proof}
Consider the function $\kappa_\mathbb{Z}:\mathbb{Z}\to\mathbb{Z}_{\ge 0}$ given by $\kappa_\mathbb{Z}(m)=\sum_{g\in G/T}\kappa(m,g)$. Then the support of $\kappa_\mathbb{Z}$ is $\{a-s,\ldots,a-1\}$. On the other hand, if $\kappa$ is $g^\#_0$-balanced, then $\kappa_\mathbb{Z}$ is $a_0$-balanced{, that is, $\kappa_\mathbb{Z}(i)=\kappa_\mathbb{Z}(-a_0-i)$, for all $i\in\mathbb{Z}$}, which implies $-a_0-(a-s)=a-1$. The result follows.
\end{proof}

Therefore, we can replace the parameters $g^\#_0$ and $\kappa$ by $g_0$ and $(\kappa_1,\ldots,\kappa_s)$. Also, since $g_0^\#(g^\#)^2\notin T$ for any $g^\#=(a-i,g)$ with $s+1\ne 2i$, condition (ii) is automatically satisfied if $s$ is even, and affects only $\kappa_{\frac{s+1}{2}}$ if $s$ is odd. Hence, we can restate conditions (i) and (ii) in terms of $\kappa_1,\ldots,\kappa_s$ as follows:

\begin{enumerate}
\item[(i')] $\kappa_i(x)=\kappa_{s-i+1}(g_0^{-1}x^{-1})$ for all $x\in G/T$ and all $i$;
\item[(ii')] either $s$ is even or $s$ is odd and $\kappa_{\frac{s+1}{2}}(gT)$ is even whenever $g_0g^2\in T$ and $\eta(g_0g^2)=-1$ for some $g\in G$. 
\end{enumerate}

Note that condition (i') implies that $n_i=|\kappa_i|\ell=|\kappa_{s-i+1}|\ell=n_{s-i+1}$, so Type II gradings on $U(\mathscr{F})_0$ can exist only if $n_i=n_{s-i+1}$ for all $i$, as expected from the structure of the automorphism group (see Lemma \ref{aut_Lie}).

Let us describe explicitly a Type II grading on $U(\mathscr{F})_0$ in the isomorphism class parametrized by $(T,\beta,g_0,(\kappa_1,\ldots,\kappa_s))$. For each $1\le i<\frac{s+1}{2}$, we fill two $|\kappa_i|$-tuples, $\gamma_i$ and $\gamma_{s-i+1}$, simultaneously as follows, going from left to right in $\gamma_i$ and from right to left in $\gamma_{s-i+1}$. For each coset $x\in G/T$ that lies in the support of $\kappa_i$, we choose an element $g\in x$ and place $\kappa_i(x)$ copies of $g$ into $\gamma_i$ and as many copies of $g_0^{-1}g^{-1}$ into $\gamma_{s-i+1}$. If $s$ is odd, we fill the middle $|\kappa_i|$-tuple $\gamma_i$, with $i=\frac{s+1}{2}$, in the following manner: $\gamma_i$ will be the concatenation of (possibly empty) tuples $\gamma^\triangleleft$, $\gamma^+$, $\gamma^0$, $\gamma^-$ and $\gamma^\triangleright$ (in this order), where $\gamma^\triangleleft$ and $\gamma^+$ are to be filled from left to right, $\gamma^-$ and $\gamma^\triangleright$ from right to left, and $\gamma^0$ in any order. For each $x$ in the support of $\kappa_i$, we choose an element $g\in x$. If $g_0g^2\notin T$, we place $\kappa_i(x)$ copies of $g$ into $\gamma^\triangleleft$ and as many copies of $g_0^{-1}g^{-1}$ into $\gamma^\triangleright$. If $g_0g^2\in T$ and $\eta(g_0g^2)=-1$, we place $\frac12\kappa_i(x)$ copies of $g$ in each of $\gamma^+$ and $\gamma^-$. Finally, if $g_0g^2\in T$ and $\eta(g_0g^2)=1$, we place $\kappa_i(x)$ copies of $g$ into $\gamma^0$. Concatenating these $\gamma_1,\ldots,\gamma_s$ results in a $k$-tuple $\gamma=(g_1,\ldots,g_k)$ of elements of $G$. Taking them modulo $\langle f\rangle$, we define a $\overline{G}$-grading on $M_k$ and, consequently, on $M_n\simeq M_k\otimes D$, so $M_n=\bigoplus_{\bar{g}\in\overline{G}}R_{\bar{g}}$. Then we construct a matrix $\Phi\in M_k(D)\simeq M_k\otimes D$ as follows:
\begin{equation}\label{Phi}
\begin{split}
\Phi&=\mathrm{diag}(\chi(g_1^{-1})I_\ell,\ldots,\chi(g_p^{-1})I_\ell)\oplus\mathrm{diag}(X_{\bar{g}_0\bar{g}_{p+1}^2},\ldots,
X_{\bar{g}_0\bar{g}_{p+q}^2})\\
&\oplus\widetilde{\mathrm{diag}}(X_{\bar{g}_0\bar{g}_{p+q+1}^2},\ldots,X_{\bar{g}_0\bar{g}_{k-p-q}^2})\\
&\oplus\mathrm{diag}(-X_{\bar{g}_0\bar{g}_{k-p-q+1}^2},\ldots,-X_{\bar{g}_0\bar{g}_{k-p}^2})\oplus\mathrm{diag}(\chi(g_{k-p+1}^{-1})I_\ell,\ldots,\chi(g_k^{-1})I_\ell),
\end{split}
\end{equation}
where $p$ is the sum of the lengths of $\gamma_1,\ldots,\gamma_{\lfloor\frac{s}2\rfloor}$, and $\gamma^\triangleleft$, $q$ is the length of $\gamma^+$, and $\widetilde{\mathrm{diag}}$ denotes arrangement of entries along the secondary diagonal (from left to right).
Finally, we use $\Phi$ to define a $G$-grading on $M_n^{(-)}$:
\begin{equation}\label{Type2Lie}
M_n^{(-)}=\bigoplus_{g\in G}R_g\text{ where }R_g=\{X\in R_{\bar{g}}\mid\Phi^{-1}X^\tau\Phi=-\chi(g)X\},
\end{equation}
which restricts to the desired grading on $U(\mathscr{F})_0$.
 
Thus we obtain the following classification of $G$-gradings on $U(\mathscr{F})_0$ from our Theorem \ref{th:main_Lie} and the known classification for $\mathfrak{sl}_n$ (as stated in \cite[Theorem 45]{BKE2018} and \cite[Theorem 3.53]{EK2013}).

\begin{Cor}\label{cor:main_Lie}
Every grading on $U(\mathscr{F})_0$ by an abelian group $G$ is isomorphic either to a Type I grading with parameters $(T,\beta,(\kappa_1,\ldots,\kappa_s))$, where $|\kappa_i|=n_i\sqrt{|T|}$, or to a Type II grading with parameters $(T,\beta,g_0,(\kappa_1,\ldots,\kappa_s))$, where $|\kappa_i|\sqrt{|T|/2}=n_{i}$ and $T$ is $2$-elementary. Type II gradings can occur only if $n>2$ and $n_i=n_{s-i+1}$ for all $i$, and their parameters are subject to the conditions (i') and (ii') above. Moreover, gradings of Type I are not isomorphic to gradings of Type II, and within each type we have the following:
\begin{enumerate}
\item[(I)] $(T,\beta,(\kappa_1,\ldots,\kappa_s))$ and $(T',\beta',(\kappa_1',\ldots,\kappa_s'))$ determine the same isomorphism class if and only if $T'=T$ and there exists $g\in G$ such that either $\beta'=\beta$ and $\kappa_i'=g\kappa_i$ for all $i$, or $n>2$, $\beta'=\beta^{-1}$ and $\kappa_i'=g\bar{\kappa}_{s-i+1}$ for all $i$, where $\bar{\kappa}(x):=\kappa(x^{-1})$ for all $x\in G/T$.
\item[(II)] $(T,\beta,g_0,(\kappa_1,\ldots,\kappa_s))$ and $(T',\beta',g_0',(\kappa_1',\ldots,\kappa_s'))$ determine the same isomorphism class if and only if $T'=T$, $\beta'=\beta$, and there exists $g\in G$ such that $g'_0=g^{-2}g_0$ and $\kappa'_i=g\kappa_i$ for all $i$.\qed
\end{enumerate}
\end{Cor}

\section{Commutativity of the grading group}\label{commut_supp}
Our immediate goal is to prove Lemma \ref{can_Lie}. The arguments will work without assuming a priori that the grading group is abelian, and, in fact, our second goal will be to prove that the elements of the support of any group grading on $U(\mathscr{F})_0$ must commute with each other. It will be more convenient to make computations in $U(\mathscr{F})^{(-)}$. So, suppose $U(\mathscr{F})^{(-)}$ is graded by an arbitrary group $G$. We still assume that $\mathrm{char}\,\mathbb{F}=0$, but $\mathbb{F}$ need not be algebraically closed. 

Write $U(\mathscr{F})=\bigoplus_{1\le i\le j\le s}B_{ij}$, where each $B_{ij}$ is the set of matrices with non-zero entries only in the $(i,j)$-th block. Thus, $J_m=B_{1,m+1}\oplus B_{2,m+2}\oplus\cdots\oplus B_{s-m,s}$ for all $m\in\{0,1,\ldots,s-1\}$. It is important to note that $[J_1,J_{m}]=J_{m+1}$ and hence the Lie powers of the Jacobson radical $J=\bigoplus_{m>0}J_m$ coincide with its associative powers.

Let $e_i\in B_{ii}$ be the identity matrix of each diagonal block and let 
\[
\mathfrak{d}=\mathrm{Span}\{e_1,e_2,\ldots,e_s\}.
\] 
We can write $B_{ii}=\mathfrak{s}_i\oplus \mathbb{F}e_i$, where $\mathfrak{s}_i=[B_{ii},B_{ii}]\simeq\mathfrak{sl}_{n_i}$. Let $S=\bigoplus_{i=1}^s\mathfrak{s}_{i}$ and $R=\mathfrak{d}\oplus J$. Then $U(\mathscr{F})^{(-)}=S\oplus R$ is a Levi decomposition.

We will need the following graded version of Levi decomposition, which was established in \cite{PRZ2013} and then improved in \cite{Gord2016} by weakening the conditions on the ground field:

\begin{Thm}[{\cite[Corollaries 4.2 and 4.3]{Gord2016}}]\label{thm_gord}
Let $L$ be a finite-dimensional Lie algebra over a field $\mathbb{F}$ of characteristic $0$, graded by an arbitrary group $G$. 
Then the radical $R$ of $L$ is graded and there exists a maximal semisimple subalgebra $B$ such that $L=B\oplus R$ (direct sum of graded subspaces).\qed
\end{Thm}

\begin{Cor}\label{Levi}
Consider any $G$-grading on $U(\mathscr{F})^{(-)}$. Then the ideal $R$ is graded. Moreover, there exists an isomorphic $G$-grading on $U(\mathscr{F})^{(-)}$ such that $S$ is also graded.
\end{Cor}

\begin{proof}
By Theorem \ref{thm_gord}, there exists a graded Levi decomposition $U(\mathscr{F})^{(-)}=B\oplus R$. But $U(\mathscr{F})^{(-)}=S\oplus R$ is another Levi decomposition, so, by Malcev's Theorem (see e.g. \cite[Corollary 2 on p.~93]{Jac1979}), there exists an (inner) automorphism $\psi$ of $U(\mathscr{F})^{(-)}$ such that $\psi(B)=S$.  Applying $\psi$ to the given $G$-grading on $U(\mathscr{F})^{(-)}$, we obtain a new $G$-grading on $U(\mathscr{F})^{(-)}$ with respect to which $S$ is graded.
\end{proof}

\begin{Lemma}\label{part_diag}
For any $G$-grading on $U(\mathscr{F})^{(-)}$, there exists an isomorphic $G$-grading such that the subalgebras $\mathfrak{d}$ and $S$ are graded.
\end{Lemma}

\begin{proof}
We partition $\{1,\ldots,s\}=\{i_1,\ldots,i_r\}\cup\{j_1,\ldots,j_{s-r}\}$ so that $n_{i_k}=1$ and $n_{j_k}>1$. Denote $e_{\triangle}=\sum_{k=1}^r e_{i_k}$, then $e_{\triangle}U(\mathscr{F})e_{\triangle}\simeq UT_r$, the algebra of upper triangular matrices (if $r>0$).

By Corollary \ref{Levi}, we may assume that $S$ is graded. Then its centralizer in $R$, $N:=\mathrm{C}_R(S)$, is a graded subalgebra. It coincides with $\mathrm{Span}\{e_{j_1},\ldots,e_{j_t}\}\oplus e_{\triangle}U(\mathscr{F})e_{\triangle}$, and its center (which is also graded) coincides with $\mathrm{Span}\{e_{j_1},\ldots,e_{j_t},e_{\triangle}\}$. 
If $r=0$, then $N=\mathfrak{d}$ and we are done. Assume $r>0$. Then we obtain a $G$-grading on $N/\mathfrak{z}(N)\simeq UT_r^{(-)}/\mathbb{F} {1}\simeq (UT_r)_0$.
These gradings were classified in \cite{pkfy2017}, where it was shown that, after applying an automorphism of $UT_r^{(-)}$, the subalgebra of diagonal matrices in $UT_r^{(-)}$ is graded. Since $-\tau$ preserves this subalgebra, we may assume that the automorphism in question is inner. But an inner automorphism of $e_{\triangle}U(\mathscr{F})e_{\triangle}$ can be extended to an inner automorphism of $U(\mathscr{F})$. {Indeed, let $y$ be an invertible element of $e_{\triangle}U(\mathscr{F})e_{\triangle}$.} Then $x=\sum_{k=1}^{s-r} e_{j_k}+y\in U(\mathscr{F})^\times$ and $\mathrm{Int}(x)$ extends $\mathrm{Int}(y)$. Moreover, $\mathrm{Int}(x)$ preserves $S$. Therefore, we may assume that the subalgebra of diagonal matrices in $N/\mathfrak{z}(N)$ is graded. But the inverse image of this subalgebra in $N$ is precisely $\mathfrak{d}$, so $\mathfrak{d}$ is graded.
\end{proof}

It will be convenient to use the following technical concept:

\begin{Def}
Let $L$ be a $G$-graded Lie algebra. We call $x\in L$ \textit{semihomogeneous} if $x=x_h+x_z$, with $x_h$ homogeneous and $x_z\in\mathfrak{z}(L)$. If $x_h\notin\mathfrak{z}(L)$, we define the \emph{degree} of $x$ as $\deg x_h$ and denoted it by $\deg x$.
\end{Def}

An important observation is that if $x$ and $y$ are semihomogeneous and $[x,y]\ne 0$, then $[x,y]$ is homogeneous of degree $\deg x\deg y$ (as $[x,y]$ will coincide with $[x_h,y_h]$).

\begin{Prop}\label{diagonal}
For any $G$-grading on $U(\mathscr{F})^{(-)}$, there exists an isomorphic $G$-grading with the following properties: 
\begin{enumerate}
\item[(i)] the subalgebras $\mathfrak{s}_{k}+\mathfrak{s}_{s-k+1}$ are graded,
\item[(ii)] the elements $e_k-e_{s-k+1}$ ($k\ne\frac{s+1}{2}$) are semihomogeneous of degree $1_G$, and
\item[(iii)] the elements $e_k+e_{s-k+1}$ are semihomogeneous of degree $f$ (if $s>2$), where $f\in G$ is an element of order at most $2$. 
\end{enumerate}
\end{Prop}

\begin{proof}
By Lemma \ref{part_diag}, we may assume that $S$ and $\mathfrak{d}$ are graded subalgebras. Also note that $J=[R,R]$ and all of its powers are graded ideals. We proceed by induction on $s$. If $s=1$, then $\mathfrak{s}_1=S$ is graded and there is nothing more to prove. If $s=2$, then $\mathfrak{s}_{1}\oplus\mathfrak{s}_{2}=S$ is graded. Also,  $\mathrm{Span}\{e_1,e_2\}=\mathfrak{d}$ and $e_1+e_2={1}$ is central, so $e_1-e_2$ is a semihomogeneous element. Its degree must be equal to $1_G$, because $[e_1-e_2,x]=2x$ for any $x\in J=B_{12}$. Now assume $s>2$. 

\textbf{Claim 1}: $N:=B_{11}\oplus B_{ss}\oplus\mathbb{F}{1}\oplus J$ is graded.

First suppose $s\ge 4$. Consider $J^{s-2}=J_{s-2}\oplus J_{s-1}$ (the three blocks in the top right corner) and the graded ideal $C:=\mathrm{C}_R(J^{s-2})=R\cap\mathrm{C}_{U(\mathscr{F})^{(-)}}(J^{s-2})$. It is easy to see that
\[
C=\mathrm{Span}\{e_2,\ldots,e_{s-1}\}\oplus\mathbb{F}{1}\oplus B_{23}\oplus\cdots\oplus B_{s-2,s-1}\oplus J^2.
\]
Now, the adjoint action induces on $C/J^2$ a natural structure of a graded $U(\mathscr{F})^{(-)}$-module, and one checks that  $N=\mathrm{Ann}_{U(\mathscr{F})^{(-)}}(C/J^2)+J$, so $N$ is graded.

If $s=3$, then consider $J^2=J_2=B_{13}$ and the graded ideal $\tilde{C}:=\mathrm{C}_{U(\mathscr{F})^{(-)}}(J^2)$. One checks that
\[
\tilde{C}=B_{22}\oplus\mathbb{F}{1}\oplus J,
\]
and hence $N=\mathrm{Ann}_{U(\mathscr{F})^{(-)}}(\tilde{C}/J)$. This completes the proof of Claim 1.

It follows that $S\cap N=\mathfrak{s}_{1}\oplus\mathfrak{s}_{s}$ is a graded subalgebra, and 
\[
I_1:=\mathfrak{d}\cap N=\mathrm{Span}\{e_1,e_s,{1}\}
\]
is graded as well.  Hence, $\mathrm{C}_{I_1}(J^{s-1})=\mathrm{Span}\{e_1+e_s,{1}\}$ is graded, so we conclude that $e_1+e_s$ is semihomogeneous. Denote its degree by $f$.

\textbf{Claim 2}: $f^2=1_G$ and $e_1-e_s$ is semihomogeneous of degree $1_G$.

Since $I_1/\mathbb{F}{1}$ is spanned by the images of $e_1$ and $e_s$, there must exists a semihomo\-ge\-neous linear combination $\tilde{e}$ of $e_1$ and $e_s$ that is not a scalar multiple of $e_1+e_s$. Consider the graded $I_1$-module $J^{s-2}/J^{s-1}$. As a module, it is isomorphic to $B_{1,s-1}\oplus B_{2,s}$, where ${1}$ acts as $0$, $e_1$ as the identity on the first summand and $0$ on the second, and $e_s$ as $0$ on the first and the negative identity on the second. Using this isomorphism, we will write the elements $x\in J^{s-2}/J^{s-1}$ as $x=x_1+x_2$ with $x_1\in B_{1,s-1}$ and $x_2\in B_{2,s}$. Since the situation is symmetric in $e_1$ and $e_s$, we may assume without loss of generality that $\tilde{e}=e_1+\alpha e_s$, $\alpha\ne 1$. Pick a homogeneous element $x=x_1+x_2$ with $x_1\ne 0$. First, we observe that $(e_1+e_s)\cdot((e_1+e_s)\cdot x)=x$, which implies $f^2=1_G$. If $x_2=0$, then $\tilde{e}\cdot x=(e_1+e_2)\cdot x=x$, and this implies that the semihomogeneous elements $\tilde{e}$ and $e_1+e_2$ both have degree $1_G$, which proves the claim. If $\alpha=0$, then $\tilde{e}\cdot x=x_1-\alpha x_2=x_1$ is homogeneous and we can apply the previous argument. So, we may assume that $\alpha\ne 0$. 

Suppose for a moment that we have $\deg\tilde{e}=1_G$. If $\alpha=-1$, we are done. Otherwise, we can consider the homogeneus element $0\ne x+\alpha^{-1}\tilde{e}\cdot x\in B_{1,s-1}$ and apply the previous argument again. 

It remains to prove that $\deg\tilde{e}=1_G$. Denote this degree by $g$ and assume $g\ne 1_G$. Considering
\[
D:=\mathrm{Span}\{x,\tilde{e}\cdot x,\tilde{e}\cdot(\tilde{e}\cdot x),\ldots\},
\]
we see, on the one hand, that $\dim D\le2$, because $D\subset\mathrm{Span}\{x_1,x_2\}$. On the other hand, non-zero homogeneous elements of distinct degrees are linearly independent, so the order of $g$ does not exceed $2$. By our assumption, it must be equal to $2$. Then $x$ and $\tilde{e}\cdot x$ form a basis of $D$ and $y:=\tilde{e}\cdot(\tilde{e}\cdot x)$ has the same degree as $x$. Therefore, $y=\lambda x$ for some $\lambda\ne0$. On the other hand, $y=x_1+\alpha^2 x_2$, hence $\alpha=\pm1$. The case $\alpha=1$ is excluded, whereas $\alpha=-1$ implies $\tilde{e}\cdot x=x$, which contradicts $g\ne 1_G$. The proof of Claim 2 is complete.

We have established all assertions of the proposition for $k=1$. We are going to use the induction hypothesis for $k>1$. To this end, let $e:={1}-(e_1+e_s)$ and consider $eU(\mathscr{F})e\simeq UT(n_2,\ldots,n_{s-1})$. Observe that the operator $\mathrm{ad}(e_1-e_s)$ on $U(\mathscr{F})^{(-)}$ preserves degree and acts as $0$ on $B_{11}\oplus eU(\mathscr{F})e\oplus B_{ss}$, as the identity on the blocks $B_{12},\ldots,B_{1,s-1}$ and $B_{2s},\ldots,B_{s-1,s}$ and as $2$ times the identity on $B_{1s}$. It follows that 
\begin{align*}
T_1:&=\Big(\mathrm{id}-\frac12\mathrm{ad}(e_1-e_s)\Big)\Big(\mathrm{id}-\mathrm{ad}(e_1-e_s)\Big)U(\mathscr{F})^{(-)}\\
&=B_{11}\oplus eU(\mathscr{F})e\oplus B_{ss},
\end{align*}
is a graded subspace.
Hence, $L_1:=\mathrm{C}_{T_1}(J^{s-1})=\mathbb{F}(e_1+e_s)\oplus eU(\mathscr{F})e$ is graded and we can apply the induction hypothesis to $L_1/\mathbb{F}(e_1+e_s)\simeq UT(n_2,\ldots,n_{s-1})$. Therefore, for $1<k\le\frac{s+1}{2}$, the subalgebras $\mathbb{F}(e_1+e_s)\oplus(\mathfrak{s}_k+\mathfrak{s}_{s-k+1})\subset L_1$ are graded, the elements $e_k+e_{s-k+1}$ are semi\-homo\-ge\-neous of degree $f'$ in $L_1$ (if $s>4$), and the elements $e_k-e_{s-k+1}$ ($k\ne\frac{s+1}{2}$) are semihomogeneous of degree $1_G$ in $L_1$. {For the subalgebras, we can get rid of the unwanted term $\mathbb{F}(e_1+e_s)$ by passing to the derived algebra, so we conclude that $\mathfrak{s}_k+\mathfrak{s}_{s-k+1}$ are graded. For the elements, since $\mathfrak{z}(L_1)=\mathbb{F}(e_1+e_s)\oplus\mathbb{F}1$, we also have to get rid of $\mathbb{F}(e_1+e_s)$ before we can conclude that they are semihomogeneous in $U(\mathscr{F})^{(-)}$.}

\textbf{Claim 3}: $e_k+e_{s-k+1}$ are semihomogeneous of degree $f$ in $U(\mathscr{F})^{(-)}$.

If $s=3$, then $e_2={1}-(e_1+e_3)$ is semihomogeneous of degree $f$. If $s=4$, then $e_2+e_{s-1}={1}-(e_1+e_s)$ is semihomogeneous of degree $f$. So, assume $s>4$. By the above paragraph, we know there exist $\alpha_k$ such that $\alpha_k(e_1+e_s)+e_k+e_{s-k+1}$ are semihomogeneous of degree $f'$ in $U(\mathscr{F})^{(-)}$. If $\alpha_2=0$, then pick a non-zero homogeneous element $x\in J^{s-2}/J^{s-1}$. Since $(e_1+e_s)\cdot x=-(e_2+e_{s-1})\cdot x\ne 0$, we conclude that $f=f'$ and the claim follows, because we can subtract the scalar multiples of $e_1+e_s$ from the elements $\alpha_k(e_1+e_s)+e_k+e_{s-k+1}$. If $\alpha_2\ne0$, consider instead the graded  $U(\mathscr{F})^{(-)}$-module $([e_1-e_s,J^2]+J^3)/J^3$. As a module, it is isomorphic to $B_{13}\oplus B_{s-2,s}$, so $e_2+e_{s-1}$ annihilates it. Picking a non-zero homogeneous element $x$, we get 
\[
(\alpha_2(e_1+e_s)+e_2+e_{s-1})\cdot x=\alpha_2(e_1+e_s)\cdot x\ne0, 
\]
so again $f=f'$ and the claim follows. 

\textbf{Claim 4}: $e_k-e_{s-k+1}$ are semihomogeneous of degree $1_G$ in $U(\mathscr{F})^{(-)}$.

We know there exist $\alpha'_k$ such that $\alpha'_k(e_1+e_s)+e_k-e_{s-k+1}$ are semihomogeneous of degree $1_G$ in $U(\mathscr{F})^{(-)}$. If $f=1_G$, then we can subtract the scalar multiples of $e_1+e_s$, so we are done. If $f\ne 1_G$, we want to prove that $\alpha'_k=0$. By way of contradiction, assume $\alpha'_k\ne 0$. If $k<\frac{s}{2}$, then $e_k-e_{s-k+1}$ annihilates the graded module $([e_1-e_s,J^k]+J^{k+1})/J^{k+1}$, so, using the argument in the proof of Claim~3, we conclude that $\deg(e_1+e_s)=1_G$, a contradiction. It remains to consider the case $s=2k$. If $s>4$, then $e_{s/2}-e_{s/2+1}$ annihilates the graded module $([e_1-e_s,J]+J^2)/J^2$, which is isomorphic to $B_{12}\oplus B_{s-1,s}$, so the same argument works. If $s=4$, then $e_2-e_3$ does not annihilate this module, but acts on it as the negative identity. Picking a non-zero homogeneous element $x$, we get
\[
x+(\alpha'_2(e_1+e_s)+e_2-e_3)\cdot x=\alpha'_2(e_1+e_s)\cdot x\ne 0,
\]
so again $\deg(e_1+e_s)=1_G$, a contradiction.

The proof of the proposition is complete.
\end{proof}

\begin{proof}[Proof of Lemma \ref{can_Lie}]
We extend a given $G$-grading on $U(\mathscr{F})_0$ to $U(\mathscr{F})^{(-)}$ by defining the degree of ${1}$ an arbitrarily. Then $U(\mathscr{F})_0\simeq U(\mathscr{F})^{(-)}/\mathbb{F}{1}$ as a graded algebra.
By Lemma \ref{part_diag}, we may assume that $\mathfrak{d}$ and $S$ are graded, hence the subalgebra $J_0=\mathfrak{d}\oplus S$ and  its homomorphic image $J_0/\mathbb{F}{1}\simeq J_0\cap U(\mathscr{F})_0$ in $U(\mathscr{F})_0$ are graded. (In fact, by Proposition \ref{diagonal}, we can say more: every subalgebra $B_{ii}+B_{s-i+1}+\mathbb{F}{1}$ is graded.) To deal with $J_m$ for $m>0$, we will use the semihomogeneous elements $d_i:=e_i-e_{s-i+1}$ of degree $1_G$ ($i\ne \frac{s+1}{2}$). Fix $i<j$. If $i+j\ne s+1$, then 
\[
B_{ij}\oplus B_{s-j+1,s-i+1}=\mathrm{ad}(d_i-d_j)\mathrm{ad}(d_i)\mathrm{ad}(d_j)U(\mathscr{F})^{(-)},
\]
which is a graded subspace. If $i+j=s+1$, then 
\[
B_{ij}=(\mathrm{id}-\mathrm{ad}(d_i))\mathrm{ad}(d_i)J^{s-i+1}
\]
is graded. Thus, $B_{ij}+B_{s-j+1,s-i+1}$ is graded for all $i<j$, hence so is $J_m$.
\end{proof}

Now, we proceed to prove that the support of any $G$-grading on $U(\mathscr{F})_0$ is a \emph{commutative subset} of $G$ in the sense that its elements commute with each other. The key observation is that, if $x$ and $y$ are homogeneous elements in any $G$-graded Lie algebra and $[x,y]\ne 0$, then $\deg x$ must commute with $\deg y$. By induction, one can generalize this as follows: if $x_1,\ldots,x_k$ are homogeneous and $[\ldots[x_1,x_2],\ldots,x_k]\ne 0$ then the degrees of $x_i$ must commute pair-wise. This fact was used to show that the support of any graded-simple Lie algebra is commutative (see e.g. \cite[Proposition 2.3]{PRZ2013} or the proof of Proposition 1.12 in \cite{EK2013}). We will need the following two lemmas.

\begin{Lemma}\label{cyclic}
Suppose a semidirect product of Lie algebras $V\rtimes L$ is graded by a group $G$ in such a way that both the ideal $V$ and the subalgebra $L$ are graded. Assume that the support of $L$ is commutative and, as an $L$-module, $V$ is faithful and generated by a single homogeneous element. Then the support of $V\rtimes L$ is commutative.
\end{Lemma}

\begin{proof}
Let $v$ be a homogeneous generator of $V$ as an $L$-module and let $g=\deg v$. Denote by $H$ the abelian subgroup generated by $\mathrm{Supp}\,L$. Then $\mathrm{Supp}\,V$ is contained in the coset $Hg$. In particular, the subgroup generated by $\mathrm{Supp}\,(V\rtimes L)$ is also generated by $H$ and $g$, so it is sufficient to prove that $g$ commutes with all elements of $\mathrm{Supp}\,L$. Let $a\ne 0$ be a homogeneous element of $L$. Since $V$ is faithful, there exists a homogeneous element $w\in V$ such that $a\cdot w\ne 0$. But, in the semidirect product,  $a\cdot w=[a,w]$, hence $\deg a$ and $\deg w$ commute. Since $\deg a\in H$, $\deg w\in Hg$, and $H$ is abelian, we conclude that $\deg a$ commutes with $g$. 
\end{proof}

{
\begin{Lemma}\label{cyclic2}
Let $L=L_1\times\cdots\times L_k$ and suppose the semidirect product $V\rtimes L$ is graded by a group $G$ in such a way that $V$ and each subalgebra $L_i$ are graded. Assume that $V$ is graded-simple as an $L$-module and, for each $i$, $\mathrm{Supp}\,L_i$ is commutative and $V$ is faithful as an $L_i$-module. Then the support of $V\rtimes L$ is commutative.
\end{Lemma}

\begin{proof}
One checks that, if we redefine the bracket on the ideal $V$ to be zero while keeping the same bracket on the subalgebra $L$ and the same $L$-module structure on $V$, the resulting semidirect product is still $G$-graded, so we may suppose $[V,V]=0$.
Let $v$ be any non-zero homogeneous element of $V$ (hence a generator of $V$ as an $L$-module). Let $W_i$ be the $L_i$-submodule generated by $v$. Since the actions of $L_i$ and $L_j$ on $V$ commute with each other for all $j\ne i$, $W_i$ must be a faithful $L_i$-module, so we can apply Lemma \ref{cyclic} to the graded subalgebra $W_i\rtimes L_i$ and conclude that $\deg v$ commutes with the elements of $\mathrm{Supp}\,L_i$ for each $i$. It remains to prove that the elements of $\mathrm{Supp}\,L_i$ commute with the elements of $\mathrm{Supp}\,L_j$ for $j\ne i$. Let $a\ne 0$ be a homogeneous element of $L_i$. Pick a homogeneous $v\in V$ such that $v':=a\cdot v\ne0$ and denote $g=\deg v$ and $g'=\deg v'$. By the previous argument, both $g$ and $g'$ commute with every element of $\mathrm{Supp}\,L_j$. But this implies that $\deg a$ commutes with every element of $\mathrm{Supp}\,L_j$. 
\end{proof}
}

\begin{Thm}\label{supp_commutativity}
The support of any group grading on $U(\mathscr{F})_0$ over a field of charac\-te\-ristic $0$ generates an abelian subgroup.
\end{Thm}
\begin{proof}
The result is known for simple Lie algebras, so we assume $s>1$.
We extend the grading to $U(\mathscr{F})^{(-)}$ and bring it to the form described in Proposition \ref{diagonal}. Then, as in the proof of Lemma \ref{can_Lie} just above, we can break $J$ into the direct sum of graded subspaces of the form $B_{ij}\oplus B_{s-j+1,s-i+1}$ ($i+j\ne s+1$) or $B_{ij}$ ($i+j=s+1$), for all $1\le i<j\le s$. Also, $\tilde{\mathfrak{s}}_i:=\mathfrak{s}_i+\mathfrak{s}_{s-i+1}$ are graded subalgebras (possibly zero). Note that any non-zero $\tilde{\mathfrak{s}}_i$ is graded-simple and, therefore, its support is commutative, except in the following situation: $i\ne\frac{s+1}{2}$ and one of the ideals $\mathfrak{s}_i$ and $\mathfrak{s}_{s-i+1}$ is graded. In this case, the other ideal is graded, too, being the centralizer of the first in $\tilde{\mathfrak{s}}_i$, and we can apply Lemma \ref{cyclic2} to the graded algebra $B_{i,s-i+1}\oplus\tilde{\mathfrak{s}}_i\simeq B_{i,s-i+1}\rtimes(\mathfrak{s}_i\times\mathfrak{s}_{s-i+1})$ to conclude that the support of $\tilde{\mathfrak{s}}_i$ is still commutative. Moreover, its elements commute with those of $\mathrm{Supp}\,B_{i,s-i+1}$, so we are done in the case $s=2$. From now on, assume $s>2$. {Let $f$ be the element of $G$ as in Proposition \ref{diagonal}.}

\textbf{Case 1}: $f=1_G$. 

Here each block $B_{ij}$ and each subalgebra $\mathfrak{s}_i$ is graded. Indeed, each element $e_i$ is semihomogeneous of degree $1_G$. If $i+j=s+1$, then we already know that $B_{ij}$ is graded, and otherwise $B_{ij}=\mathrm{ad}(e_i)(B_{ij}\oplus B_{s-j+1,s-i+1})$, so it is still graded. For $\tilde{\mathfrak{s}}_i$, it is sufficient to consider $i\le\frac{s+1}{2}$. If $i=\frac{s+1}{2}$, then we already know that $\mathfrak{s}_i$ is graded, and otherwise we can find $j>i$ such that $j\ne s-i+1$, which implies that $\mathfrak{s}_i=\mathrm{C}_{\tilde{\mathfrak{s}}_i}(B_{ij})$ is still graded. 

Applying Lemma \ref{cyclic2} to $B_{ij}\rtimes (\mathfrak{s}_i\times\mathfrak{s}_j)$, we conclude that the supports of non-zero $\mathfrak{s}_i$ and $\mathfrak{s}_j$ commute element-wise with one another and also with $\mathrm{Supp}\,B_{ij}$. (This works even if one of $\mathfrak{s}_i$ and $\mathfrak{s}_j$ is zero.) It follows that $\mathrm{Supp}\,S$ generates an abelian subgroup $H$ in $G$. It also commutes element-wise with $\mathrm{Supp}\,J$. Indeed, since $\mathrm{Supp}\,B_{ij}$ is contained in a coset of $H$, it is sufficient to prove that the degree of one non-zero homogeneous element of $B_{ij}$ commutes with the elements of $\mathrm{Supp}\,\mathfrak{s}_k$. We already know this if $k=i$ or $k=j$. Otherwise, we will have $k<i<j$, $i<k<j$ or $i<j<k$. In the last case, we have $[B_{ij},B_{jk}]=B_{ik}$, so we can find homogeneous elements $x\in B_{ij}$ and $y\in B_{jk}$ such that $0\ne [x,y]\in B_{ik}$. Since the elements of $\mathrm{Supp}\,\mathfrak{s}_k$ commute with $\deg y$ and with $\deg x\deg y$, they must commute with $\deg x$ as well. The other two cases are treated similarly. 

It remains to prove that $\mathrm{Supp}\,J$ is commutative. Since $J_1$ generates $J$ as a Lie algebra, it is sufficient to prove that, for any $1\le i\le j\le s-1$, the sets $\mathrm{Supp}\,B_{i,i+1}$ and $\mathrm{Supp}\,B_{j,j+1}$ commute with one another element-wise. But we can find homogeneous elements $x_1\in B_{12},\,x_2\in B_{23},\ldots,x_{s-1}\in B_{s-1,s}$ such that $[\ldots[x_1,x_2],\ldots,x_{s-1}]\ne 0$, so the degrees of $x_1,x_2,\ldots,x_{s-1}$ must commute pair-wise. The coset argument completes the proof of Case~1.

\textbf{Case 2}: $f\ne 1_G$. 

Here we work with $\tilde{B}_{ij}:=B_{ij}+B_{s-j+1,s-i+1}$. If $\tilde{\mathfrak{s}}_i$ and $\tilde{\mathfrak{s}}_j$ are distinct (that is, $i+j\ne s+1$) and non-zero, then $\tilde{B}_{ij}$ is a direct sum of two non-isomorphic simple $(\tilde{\mathfrak{s}}_i\times\tilde{\mathfrak{s}}_j)$-submodules. We claim that it is a graded-simple $(\tilde{\mathfrak{s}}_i\times\tilde{\mathfrak{s}}_j)$-module. Indeed, otherwise one of the submodules $B_{ij}$ and $B_{s-j+1,s-i+1}$ would be graded. But there exist scalars $\lambda_i$ such that $\tilde{e}_i:=e_i+e_{s-i+1}+\lambda_i {1}$ are homogeneous of degree $f$, and $\mathrm{ad}(\tilde{e}_i)$ acts as the identity on $B_{ij}$ and the negative identity on $B_{s-j+1,s-i+1}$, which forces $f=1_G$, a contradiction.

Therefore, we can apply Lemma \ref{cyclic2} to $\tilde{B}_{ij}\rtimes (\tilde{\mathfrak{s}}_i\times\tilde{\mathfrak{s}}_j)$ and conclude that the supports of non-zero $\tilde{\mathfrak{s}}_i$ and $\tilde{\mathfrak{s}}_j$ commute element-wise with one another, hence $\mathrm{Supp}\,S$ is commutative. 

Now consider $\tilde{B}_{ij}$, with $i+j\ne s+1$, as an $((\tilde{\mathfrak{s}}_i\times\tilde{\mathfrak{s}}_j)\times\mathbb{F}\tilde{e}_i)$-module, where one of $\tilde{\mathfrak{s}}_i$ and $\tilde{\mathfrak{s}}_j$ is allowed to be zero. The simple submodules $B_{ij}$ and $B_{s-j+1,s-i+1}$ are non-isomorphic, because they are distinguished by the action of $\tilde{e}_i$. Hence, our argument in the first paragraph shows that $\tilde{B}_{ij}$ is a graded-simple module, so we can apply Lemma \ref{cyclic2} to $\tilde{B}_{ij}\rtimes((\tilde{\mathfrak{s}}_i\times\tilde{\mathfrak{s}}_j)\times\mathbb{F}\tilde{e}_i)$ and conclude that the supports of $\tilde{\mathfrak{s}}_i$ and $\tilde{\mathfrak{s}}_j$ commute element-wise with $f$ and also with $\mathrm{Supp}\,\tilde{B}_{ij}$. Moreover, $f$ commutes with $\mathrm{Supp}\,\tilde{B}_{ij}$. If $i+j=s+1$, then $\tilde{B}_{ij}=B_{ij}$ and we can apply Lemma \ref{cyclic} to $B_{ij}\rtimes\tilde{\mathfrak{s}}_i$. 

Therefore, the elements of $\mathrm{Supp}\,S$ commute with $f$ and together generate an abelian subgroup $H$ in $G$. Then, by the same argument as in Case~1 (but using $\tilde{B}_{ij}$ instead of $B_{ij}$), we show that $\mathrm{Supp}\,S$  commutes element-wise with $\mathrm{Supp}\,J$. 
In order to prove that $f$ commutes with $\mathrm{Supp}\,J$, it is sufficient to consider $J_1$. As we have seen, $f$ commutes with $\mathrm{Supp}\,\tilde{B}_{ij}$ where $i+j\ne s+1$. The only case that is not covered in $J_1$ is $\tilde{B}_{s/2,s/2+1}=B_{s/2,s/2+1}$ for even $s$. Since $s>2$, we have $[\tilde{B}_{s/2-1,s/2},B_{s/2,s/2+1}]=\tilde{B}_{s/2-1,s/2+1}$. Since $f$ commutes with $\mathrm{Supp}\,\tilde{B}_{s/2-1,s/2}$ and with  $\mathrm{Supp}\,\tilde{B}_{s/2-1,s/2+1}$, we conclude that $f$ commutes with $\mathrm{Supp}\,B_{s/2,s/2+1}$ as well. The commutativity of $\mathrm{Supp}\,J$ is proved by the same argument as in Case~1.
\end{proof}

\section{Jordan case}\label{Jord_case}
Every Jordan isomorphism from the algebra $U(\mathscr{F})$, $s>1$, to an arbitrary associative algebra $R$ is either an associative isomorphism or anti-isomorphism \cite[Corollary 3.3]{BDW2016}. By the remark after Theorem \ref{aut_cecil}, $U(\mathscr{F})$ admits an anti-automorphism if and only if $n_{i}=n_{s-i+1}$ for all $i$. So, taking into account the structure of the automorphism group of $U(\mathscr{F})$ (see Lemma \ref{aut}), we obtain that the automorphism group of $U(\mathscr{F})^{(+)}$, that is, the algebra $U(\mathscr{F})$ viewed as a Jordan algebra with respect to the symmetrized product $x\circ y=xy+yx$, is either $\{\mathrm{Int}(x)\mid x\in U(\mathscr{F})^\times\}$ or $\{\mathrm{Int}(x)\mid x\in U(\mathscr{F})^\times\}\rtimes\langle\tau\rangle$. In both cases, the following holds:

\begin{Lemma}
{If $n>2$,} $\mathrm{Aut}(U(\mathscr{F})^{(+)})\simeq\mathrm{Aut}(U(\mathscr{F})_0)$.\qed 
\end{Lemma}

Hence, if $\mathbb{F}$ is algebraically closed of characteristic $0$ and the grading group $G$ is abelian, then the classification of $G$-gradings on the Jordan algebra $U(\mathscr{F})^{(+)}$ is equivalent to the classification of $G$-gradings on the Lie algebra $U(\mathscr{F})_0$ (see also \cite[\S 5.6]{EK2013} for the simple case, $s=1$). Thus, we get the same parametrization of the isomorphism classes of gradings as in Corollary \ref{cor:main_Lie}. The only difference is the sign in the construction of Type II gradings on $M_n^{(+)}$ (compare with Equation \eqref{Type2Lie} and recall that $\Phi$ is given by Equation \eqref{Phi}):
\begin{equation*}
M_n^{(+)}=\bigoplus_{g\in G}R_g\text{ where }R_g=\{X\in R_{\bar{g}}\mid\Phi^{-1}X^\tau\Phi=\chi(g)X\},
\end{equation*}
which are then restricted to $U(\mathscr{F})^{(+)}$. {Hence, for $n>2$, an explicit bijection between the $G$-gradings (or their isomorphism classes) on $U(\mathscr{F})^{(+)}$ and those on $U(\mathscr{F})_0$ is the following: restriction for Type I gradings and restriction with shift by the distinguished element $f$ for Type II gradings (which occur on $U(\mathscr{F})^{(+)}$ even for $n=2$, but in this case restrict to Type I gradings on $U(\mathscr{F})_0$).}

We note, however, that this result does not exclude the existence of group gradings on $U(\mathscr{F})^{(+)}$ with non-commutative support. In view of Theorem \ref{supp_commutativity}, these gradings, if they exist, are not analogous to gradings on $U(\mathscr{F})_0$.

\section{Isomorphism and practical isomorphism of graded Lie algebras}\label{practical_iso}
We use the main result of this section to obtain a classification of group gradings for $U(\mathscr{F})^{(-)}$
from the classification for $U(\mathscr{F})_0$, but it is completely general and may be of independent interest. 
Let $G$ be a group and let $L_1$ and $L_2$ be two $G$-graded Lie algebras over an arbitrary field $\mathbb{F}$.

\begin{Def}[{\cite[Definition 7]{pkfy2017}}]
$L_1$ and $L_2$ are said to be \emph{practically $G$-graded isomorphic} if there exists an isomorphism of (ungraded) algebras $\psi:L_1\to L_2$ that induces a $G$-graded isomorphism $L_1/\mathfrak{z}(L_1)\to L_2/\mathfrak{z}(L_2)$.
\end{Def}

Note that, in this case, for every homogeneous non-central $x\in L_1$, we can find $z\in\mathfrak{z}(L_1)$ such that $y=\psi(x+z)$ is homogeneous in $L_2$ and $\deg x=\deg y$.

Clearly, if $L_1$ and $L_2$ are $G$-graded isomorphic then they are practically $G$-graded isomorphic. 
The converse does not hold, but if $L_1$ and $L_2$ are practically $G$-graded isomorphic 
then the derived algebras $L_1'$ and $L_2'$ are $G$-graded isomorphic. More precisely:

\begin{Lemma}\label{iso_derived}
Assume $\psi:L_1\to L_2$ is an isomorphism of algebras that induces a $G$-graded isomorphism 
$L_1/\mathfrak{z}(L_1)\to L_2/\mathfrak{z}(L_2)$. 
Then $\psi$ restricts to a $G$-graded isomorphism $L_1'\to L_2'$.
\end{Lemma}

\begin{proof}
Let $0\ne x\in L_1'$ be homogeneous of degree $g\in G$. Then there exist in $L_1$ nonzero homogeneous $x'_i$ of degree $g'_i$ and $x''_i$ of degree $g''_i$, $i=1,\ldots,m$, such that $x=\sum_{i=1}^m[x'_i,x''_i]$ and $g'_ig''_i=g$ for all $i$. 
Also, there exist $z'_i,z''_i\in\mathfrak{z}(L_1)$ such that $\psi(x'_i+z'_i)$ is homogeneous of degree $g'_i$ and $\psi(x''_i+z''_i)$ is homogeneous of degree $g''_i$, for all $i$. Hence,
\[
\psi(x)=\psi\left(\sum_{i=1}^m[x'_i+z'_i,x''_i+z''_i]\right)=\sum_{i=1}^m[\psi(x'_i+z'_i),\psi(x''_i+z''_i)]
\]
is homogeneous in $L_2$ of degree $g$, as desired. 
\end{proof}

Now we will see what happens if we strengthen the hypothesis on $\psi$ by assuming, in addition, that it restricts to a $G$-graded isomorphism $\mathfrak{z}(L_1)\to\mathfrak{z}(L_2)$. This does not yet imply that $\psi$ itself is a $G$-graded isomorphism, but we have the following:

\begin{Thm}\label{th:main_practical}
Let $L_1$ and $L_2$ be $G$-graded Lie algebras, and assume that there exists an isomorphism of (ungraded) algebras $\psi:L_1\to L_2$ such that both the induced map $L_1/\mathfrak{z}(L_1)\to L_2/\mathfrak{z}(L_2)$ and the restriction $\mathfrak{z}(L_1)\to\mathfrak{z}(L_2)$ are $G$-graded isomorphisms. Then $L_1$ and $L_2$ are isomorphic as $G$-graded algebras.
\end{Thm}

\begin{proof}
Let $N_1\subset\mathfrak{z}(L_1)$ be a graded subspace such that
\[
\mathfrak{z}(L_1)=N_1\oplus(\mathfrak{z}(L_1)\cap L_1').
\]
By our hypothesis, $N_2:=\psi(N_1)$ is a graded subspace of $\mathfrak{z}(L_2)$. Since $L_1'\oplus N_1$ is a graded subspace of $L_1$, there exists a linearly independent set $\mathcal{B}_1=\{u_i\}_{i\in\mathscr{I}}$ of homogeneous element of $L_1$ satisfying 
\[
L_1=L_1'\oplus N_1\oplus\mathrm{Span}\,\mathcal{B}_1. 
\]
By our hypothesis, we can find $z_i\in\mathfrak{z}(L_1)$ such that $\psi(u_i+z_i)$ is a homogeneous element of $L_2$ that has the same degree as $u_i$. Since $\mathfrak{z}(L_1)\subset L_1'\oplus N_1$, the set $\mathcal{B}_2:=\{\psi(u_i+z_i)\}_{i\in\mathscr{I}}$ is linearly independent and satisfies 
\[
L_2=L_2'\oplus N_2\oplus\mathrm{Span}\,\mathcal{B}_2.
\]
Now define a linear map $\theta:L_1\to L_2$ by setting $\theta|_{L_1'\oplus N_1}=0$ and $\theta(u_i)=\psi(z_i)$ for all $i\in\mathscr{I}$. This is a ``trace-like map'' in the sense that its image is contained in $\mathfrak{z}(L_2)$ and its kernel contains $L_1'$. It follows that $\tilde\psi:=\psi+\theta$ is an isomorphism of algebras $L_1\to L_2$. Applying Lemma \ref{iso_derived}, we see that $\psi$, and hence $\tilde{\psi}$, restricts to a $G$-graded isomorphism $L_1'\oplus N_1\to L_2'\oplus N_2$. By construction, $\tilde\psi(u_i)=\psi(u_i+z_i)$. It follows that $\tilde\psi$ is an isomorphism of $G$-graded algebras. 
\end{proof}

\begin{Cor}\label{cor:main_practical}
Let $\Gamma_1$ and $\Gamma_2$ be two $G$-gradings on a Lie algebra $L$ and consider the $G$-graded algebras 
$L_1=(L,\Gamma_1)$ and $L_2=(L,\Gamma_2)$. If $L_1/\mathfrak{z}(L_1)=L_2/\mathfrak{z}(L_2)$ and $\mathfrak{z}(L_1)=\mathfrak{z}(L_2)$ as $G$-graded algebras, then $L_1\simeq L_2$ as $G$-graded algebras.
\end{Cor}

\begin{proof}
Apply the previous theorem with $\psi$ being the identity map.
\end{proof}

\begin{Cor}\label{reduction_to_trace_0}
Let $\Gamma_1$ and $\Gamma_2$ be two $G$-gradings on $U(\mathscr{F})^{(-)}$ {and assume $\mathrm{char}\,\mathbb{F}\nmid n$.} Then $\Gamma_1$ and $\Gamma_2$ are isomorphic if and only if they assign the same degree to the identity matrix ${1}$ and induce isomorphic gradings on $U(\mathscr{F})^{(-)}/\mathbb{F} {1}\simeq U(\mathscr{F})_0$. 
\end{Cor}

\begin{proof}
The ``only if'' part is clear. For the ``if'' part, take an automorphism $\psi_0$ of $U(\mathscr{F})_0$ that sends 
the grading induced by $\Gamma_1$ to the one induced by $\Gamma_2$, extend $\psi_0$ to an automorphism $\psi$ of $U(\mathscr{F})^{(-)}=U(\mathscr{F})_0\oplus\mathbb{F} {1}$ by setting $\psi({1})={1}$, and apply the theorem.
\end{proof}


\begin{thebibliography}{22}
\bibitem{BK2010} Y. Bahturin, M. Kochetov, \emph{Classification of group gradings on simple Lie algebras of types $\mathcal{A}$, $\mathcal{B}$, $\mathcal{C}$ and $\mathcal{D}$}, Journal of Algebra, \textbf{324} (2010), 2971--2989.
\bibitem{BKE2018} Y. Bahturin, M. Kochetov, A. Rodrigo-Escudero, \emph{Gradings on classical central simple real Lie algebras}, Journal of Algebra, \textbf{506} (2018), 1--42.
\bibitem{BDW2016} C. Boboc, S. D\u asc\u alescu, L. van Wyk, \emph{Jordan isomorphisms of 2-torsionfree triangular rings}, Linear and Multilinear Algebra, \textbf{64(2)} (2016), 290--296.
\bibitem{BFD2018} A. Borges, C. Fidelis, D. Diniz, \emph{Graded isomorphisms on upper block triangular matrix algebras}, Linear Algebra and its Applications, \textbf{543} (2018), 92--105.
\bibitem{Cecil} A. Cecil, \emph{Lie Isomorphisms of Triangular and Block-Triangular Matrix Algebras over Commutative Rings}. Thesis (M.Sc.), University of Victoria (Canada). 2016.
\bibitem{Cheung} W.S. Cheung, \emph{Mappings on triangular algebras}. Thesis (Ph.D.), University of Victoria (Canada). 2000.
\bibitem{EK2013} A. Elduque, M. Kochetov, \emph{Gradings on simple Lie algebras}, Mathematical Surveys and Monographs, 189. American Mathematical Society (2013).
\bibitem{Gord2016} A. S. Gordienko, \emph{Co-stability of radicals and its applications to PI-theory}, Algebra Colloquium,  \textbf{23(3)} (2016), 481--492.
\bibitem{Jac1979} N. Jacobson, \emph{Lie algebras}, republication of the 1962 original. Dover Publications (1979).
\bibitem{pkfy2017} P. Koshlukov, F. Yasumura, \emph{Group gradings on the Lie algebra of upper triangular matrices}, Journal of Algebra, \textbf{477} (2017), 294--311.
\bibitem{pkfy2017Jord} P. Koshlukov, F. Yasumura, \emph{Group gradings on the Jordan algebra of upper triangular matrices}, Linear Algebra and its Applications, \textbf{534} (2017), 1--12.
\bibitem{PRZ2013} D. Pagon, D. Repov\v{s}, M. Zaicev, \emph{Group gradings on finite dimensional Lie algebras}, Algebra Colloquium, \textbf{20(4)} (2013), 573--578.
\bibitem{MaSo1999} L. Marcoux, A. R. Sourour, \emph{Lie isomorphisms of Nest Algebras}, Journal of Functional Analysis, \textbf{164(1)} (1999), 163--180.
\bibitem{VaZa2007} A. Valenti, M. Zaicev, \textit{Group gradings on upper triangular matrices}, Archiv der Mathematik \textbf{89(1)} (2007), 33--40.
\bibitem{VaZa2012} A. Valenti, M. Zaicev, \emph{Abelian gradings on upper block triangular matrices}, Canadian Mathematical Bulletin, \textbf{55(1)}, (2012), 208--213.
\bibitem{Waterhouse} W.C. Waterhouse, \emph{Introduction to affine group schemes}, Graduate Texts in Mathematics, 66. Springer-Verlag (1979).
\bibitem{y2018} F. Yasumura, \emph{Group gradings on upper block triangular matrices}, Archiv der Mathematik, \textbf{110(4)} (2018), 327--332.
\end{thebibliography}
\end{document}